\documentclass[reqno]{amsart}
\usepackage[top=1in, bottom=1in, left=1.3in, right=1.3in]{geometry}
\usepackage{soul}
\usepackage{amsmath}
\usepackage{amssymb}
\usepackage{amsthm}
\usepackage{verbatim}
\usepackage{graphicx}
\usepackage{graphics}
\usepackage{color}
\usepackage{enumerate}
\usepackage{mathrsfs}
\usepackage{booktabs}
\usepackage[toc,page]{appendix}
\usepackage{enumerate}
\usepackage{mathrsfs}
\usepackage{fancyhdr}
\usepackage{latexsym, amsfonts, amssymb, amsmath,  amsthm}
\usepackage{amsopn, amstext, amscd,pifont}
\usepackage{amssymb,bm, cite}
\usepackage[all]{xy}
\usepackage{color}
\usepackage{blkarray}
\usepackage{hyperref}
\usepackage{float}
\usepackage{listings}
\usepackage{setspace}
\usepackage{paralist}
\usepackage{float}
\usepackage{fancyhdr}

\theoremstyle{plain}
\newtheorem{theorem}{Theorem}[section]
\newtheorem{proposition}[theorem]{Proposition}
\newtheorem{lemma}[theorem]{Lemma}
\newtheorem{corollary}[theorem]{Corollary}

\theoremstyle{definition}
\newtheorem{definition}[theorem]{Definition}
\newtheorem{example}[theorem]{Example}
\newtheorem{remark}[theorem]{Remark}

\numberwithin{equation}{section}
\graphicspath{{images/}}

\begin{document}
\title{Compactifications of the Moduli spaces of Newton Maps}
\author{Hongming Nie}
\address{Department of Mathematics, Indiana University\\
831 East Third Street, Bloomington, Indiana 47405, USA}
 \email{nieh@indiana.edu}
\date{\today}
\begin{abstract}
We study various compactifications of moduli space of Newton maps. Mainly, we focus on GIT compactifiaction and Deligne-Mumford compactification. Then we explore the relations among these compactifications.
\end{abstract}
\maketitle

\section{Introduction}
For a degree $d\ge 2$ monic polynomial $P$ with distinct roots, its Newton maps is defined by 
$$N_P(z)=z-\frac{P(z)}{P'(z)}.$$
Let $\mathrm{NM}_d$ be the space of degree $d$ Newton maps.  Note $z=\infty$ is the unique repelling fixed point for Newton maps. So  the moduli space of degree $d$ Newton maps is naturally defined by
$$\mathrm{nm}_d:=\mathrm{NM}_d/\mathrm{Aut}(\mathbb{C}).$$
In the preceding article \cite{Nie-1}, we studied the limiting behaviors of holomorphic families of Newton maps. Our main tool was Berkovich dynamics. In the present article, inspired by the recent developments in the compactifications of moduli spaces, we study various compactifications of moduli spaces $\mathrm{nm}_d$.  We also use Berkovich space to relate different compactifications.\par
The space of degree $d\ge 2$ complex rational maps, denoted by $\mathrm{Rat}_d$, is a dense open subset of $\mathbb{P}^{2d+1}$. Its moduli space $\mathrm{rat}_d:=\mathrm{Rat}_d/\mathrm{PSL}_2(\mathbb{C})$, modulo the action by conjugation, is a complex orbifold of dimension $2d-2$. To compactify the space $\mathrm{rat}_2$ of quadratic rational maps, Milnor \cite{Milnor93} considered the symmetry functions $\sigma_1,\sigma_2$ and $\sigma_3$ of the multipliers of the three fixed points and showed the pair $(\sigma_1,\sigma_2)$ parametrizes the space $\mathrm{rat}_2$. In fact, it defines an isomorphism $\mathrm{rat}_2\cong\mathbb{C}^2$ \cite[Lemma 3.1]{Milnor93}. Hence, it induces a compactification $\overline{\mathrm{rat}}_2$ of the space $\mathrm{rat}_2$ and  $\overline{\mathrm{rat}}_2\cong\mathbb{P}^2$. The boundary $\overline{\mathrm{rat}}_2\setminus\mathrm{rat}_2$ corresponds to the ideal points whose multipliers are of the form $\langle a,1/a,\infty\rangle$, where $a\in\mathbb{C}$. These ideal points can be identified with the conjugacy classes of maps $z\to az+1$. For general cases, Silverman \cite{Silverman98} applied the geometric invariant theory (GIT) to study the semistable loci $\mathrm{Rat}_d^{ss}$ and stable loci $\mathrm{Rat}_d^s$ in $\mathbb{P}^{2d+1}$ and gave compactifications $\overline{\mathrm{rat}}_d$ for $d\ge 2$ by investigating the categorical quotients and geometric quotients. The boundary $\overline{\mathrm{rat}}_d\setminus\mathrm{rat}_d$ can be identified with the equivalent classes of degenerate degree $d$ rational maps under GIT quotient.\par
For an element $[f]\in\mathrm{rat}_d$, we have $[f^n]\in\mathrm{rat}_{d^n}$. Moreover, the iterate maps 
$$\Phi_n:\mathrm{rat}_d\to\mathrm{rat}_{d^n},$$ 
sending $[f]$ to $[f^n]$, are regular for any $d\ge 2$ and $n\ge 1$. However, $\Phi_n$ can not extend continuously to maps from $\overline{\mathrm{rat}}_d$ to $\overline{\mathrm{rat}}_{d^n}$ for any $d\ge 2$ and $n\ge 2$ \cite[Section 10]{DeMarco07}. \par 
Since the space $\mathrm{NM}_d$ of degree $d$ Newton maps is a subset of $\mathrm{Rat}_d$, the moduli space $\mathrm{nm}_d$ sits in the space $\mathrm{rat}_d$. Therefore, we obtain a compactification $\overline{\mathrm{nm}}_d$ by taking the closure for $\mathrm{nm}_d$ in $\overline{\mathrm{rat}}_d$. Our first result is
\begin{theorem}\label{theorem-iterate-continuous-moduli}
The restrictions $\Phi_n|_{\mathrm{nm}_d}$ extend naturally and continuously to rational maps
$$\Phi_n|_{\overline{\mathrm{nm}}_d}:\overline{\mathrm{nm}}_d\to\overline{\mathrm{rat}}_{d^n}$$
for any $d\ge 2$ and $n\ge 2$.
\end{theorem}
Let $\overline{\mathrm{NM}}_d$ be the compactification of $\mathrm{NM}_d$ in $\mathbb{P}^{2d+1}$. All the strictly semistable points in the boundary $\partial\overline{\mathrm{NM}}_d$ have the same image under the GIT quotients, see Proposition \ref{semi-singleton}. Moreover, the iterate map $f\to f^n$ is well defined and preserves the (semi)stability of the points in the boundary of $\overline{\mathrm{NM}}_d\cap\mathrm{Rat}_d^{ss}$, see Proposition \ref{Newton-iterate-semi}. All these properties push Theorem \ref{theorem-iterate-continuous-moduli} holds. \par
As applications, we consider the compactifications by inverse limits and barycentered maximal measures, which have been introduced to resolve the indeterminacy of each rational map $\Phi_n:\overline{\mathrm{rat}}_d\dasharrow\overline{\mathrm{rat}}_{d^n}$, see \cite{DeMarco07} for details.\par 
Consider the map 
$$(\mathrm{Id},\Phi_2,\cdots,\Phi_n):\mathrm{rat}_d\to\mathrm{rat}_d\times\mathrm{rat}_{d^2}\times\cdots\times\mathrm{rat}_{d^n}.$$
Let $\Gamma_n$ be the closure of the image of $\mathrm{rat}_d$ in $\overline{\mathrm{rat}}_d\times\overline{\mathrm{rat}}_{d^2}\times\cdots\times\overline{\mathrm{rat}}_{d^n}$. There is a natural projection from $\Gamma_{n+1}$ to $\Gamma_n$. So we can take the inverse limit over $n$. We get a compactification 
$$\overleftarrow{\mathrm{rat}}_d:=\lim\limits_{\longleftarrow}\Gamma_n.$$
Via identifying $[f]$ with $([f],([f],[f^2]),\cdots)$, the space $\mathrm{rat}_d$ is a dense open subset of $\overleftarrow{\mathrm{rat}}_d$. Moreover, the map $\Phi_n:\mathrm{rat}_d\to\mathrm{rat}_{d^n}$ can extends continuously to a map from $\overleftarrow{\mathrm{rat}}_d$ to $\overleftarrow{\mathrm{rat}}_{d^n}$. To understand the structure of the space $\overleftarrow{\mathrm{rat}}_d$, DeMarco \cite{DeMarco07} studied the barycentered maximal measures and constructed a concrete model for quadratic case. For a given rational map $f\in\mathrm{Rat}_d$, the support of the maximal measure $\mu_f$ is equal to the Julia set of $f$. We can choose a representative $f$ for the conjugate class $[f]\in\mathrm{rat}_d$ such that the conformal barycenter of $\mu_f$ is at the origin. In fact, such representative is unique up to conjugation by an element in $\mathrm{SO}(3)$. Thus, it induces a continuous map
$$\Theta:\mathrm{rat}_d\to M_{bc}^1/\mathrm{SO}(3),$$
where $M_{bc}^1$ is the space of barycentered probability measures on $\widehat{\mathbb{C}}$, sending $[f]$ to $[\mu_f]$. Then considering the closure of the graph of $\Theta$, we can define a compactification
$$\widetilde{\mathrm{rat}}_d:=\overline{\mathrm{Graph}(\Theta)}\subset\overline{\mathrm{rat}}_d\times\overline{M_{bc}^1}/\mathrm{SO}(3)$$
for $\mathrm{rat}_d$.\par 
For quadratic case, the compactifications $\overleftarrow{\mathrm{rat}}_2$ and $\widetilde{\mathrm{rat}}_2$ are canonically homeomorphic \cite[Theorem 1.1]{DeMarco07}. But it is not true in general, see \cite{DeMarco07}. For Newton maps, we can define the corresponding compactifications $\overleftarrow{\mathrm{nm}}_d$ and $\widetilde{\mathrm{nm}}_d$. And we show
\begin{theorem}\label{theorem-inverse-limit-compactify}
\textit{For any $d\ge 2$, then compactifications $\overleftarrow{\mathrm{nm}}_d$ and $\widetilde{\mathrm{nm}}_d$ are both canonically homeomorphic to the compactification $\overline{\mathrm{nm}}_d$.}
\end{theorem}
The space is better behaved if we mark points. Generally, for rational maps, we can mark either the fixed points, critical points, or both. For the quadratic case, rational maps with marked points were studied in \cite{Milnor93}.
Since the Newton map $N$ is uniquely determined by the roots $\{r_1,\cdots,r_d\}$. We choose to mark these roots. In general, a degree $d\ge 2$ rational map with marked fixed points is 
$$(f,p_1,\cdots,p_{d+1})\in\mathrm{Rat}_d\times(\widehat{\mathbb{C}})^{d+1},$$
where $f\in\mathrm{Rat}_d$ and $p_1,\cdots, p_{d+1}$ are the fixed points of $f$, listed with multiplicities. Denoted by $\mathrm{Rat}_d^{\mathrm{fm}}$ the space of degree $d$ rational maps with marked fixed points and denote by $\mathrm{NM}_d^{\mathrm{fm}}$ the space of degree $d$ Newton maps with marked fixed points. Note $z=\infty$ is the unique repelling fixed point for each Newton map. Thus we consider the subset $\mathrm{NM}_d^\ast$ consisting of $(N,\infty,r_1,\cdots,r_d)\in\mathrm{NM}_d^{\mathrm{fm}}$. The $\mathrm{PSL}_2(\mathbb{C})$-action on $\mathrm{Rat}_d^{\mathrm{fm}}$ induces an action on $\mathrm{NM}_d^\ast$. Since the point $z=\infty$ is special, then it induces an $\mathrm{Aut}(\mathbb{C})$-action on $\mathrm{NM}_d^\ast$. We define the moduli space 
$$\mathrm{nm}_d^\ast:=\mathrm{NM}_d^\ast/\mathrm{Aut}(\mathbb{C}).$$
Thus the natural map $\mathrm{nm}_d^\ast\to\mathrm{nm}_d$, sending $[(N,\infty,r_1,\cdots,r_d)]$ to $N$, is $(d-2)!$ to $1$.\par 
Let $\mathcal{M}_{0,n}$ be the moduli space of Riemann surface of genus $0$ with $n$ marked points. By adding the stable curves with nodes, the moduli space $\mathcal{M}_{0,n}$ has Deligne-Mumford compactification $\widehat{\mathcal{M}}_{0,n}$ consisting of stable marked trees of complex projective lines. Note $\mathcal{M}_{0,d+1}$ is homeomorphic to $\mathrm{nm}_d^\ast$. It induces a compactification 
$$\widehat{\mathrm{nm}}^\ast_d:=\widehat{\mathcal{M}}_{0,d+1}$$
for the space $\mathrm{nm}_d^\ast$. Thus $\widehat{\mathrm{nm}}^\ast_d$ is a smooth irreducible projective variety. Our next result relates the compactifications $\widehat{\mathrm{nm}}^\ast_d$ and $\overline{\mathrm{nm}}_d$.
\begin{theorem}\label{theorem-Deligne-Mumford-GIT}
\textit{For $d\ge 2$, the natural map $\mathrm{nm}^\ast_d\to\mathrm{nm}_d$ extends  continuously to a map
$\widehat{\mathrm{nm}}^\ast_d\to\overline{\mathrm{nm}}_d$.}
\end{theorem}
We mention here that for the extension map, the fiber of a point in the boundary $\partial\overline{\mathrm{nm}}_d$ may have infinite cardinality. To construct such extension map, we define the marked Berkovich trees of spheres based on the Berkovich dynamics of Newton maps in \cite{Nie-1}. Then we define an equivalent relations on the space of marked Berkovich trees of spheres. Let $\mathcal{T}_d$ be the resulting quotient space. Associated with suitable topology, the space $\mathcal{T}_d$ is homeomorphic to $\widehat{\mathrm{nm}}^\ast_d$. For each marked Berkovich tree of spheres induced by a holomorphic family of Newton maps, it either contains a unique sphere where the subalgebraic reduction is stable, or contains spheres where the subalgebraic reductions are semistable but not stable, see Lemma \ref{Berkovich-tree-semistable-reduction}. This induces a continuous map from $\mathcal{T}_d$ to $\overline{\mathrm{nm}}_d$.

\subsection*{Outline}
This paper is organized as follows. In section \ref{background}, we state some background about the (semi)stable locus in $\mathbb{P}^{2d+1}$ and the probability measures associated to (degenerate) rational maps. In section \ref{GIT}, we study the GIT compactification $\overline{\mathrm{nm}}_d$ and prove Theorem \ref{theorem-iterate-continuous-moduli}. In section \ref{applications}, we discuss compactifications $\overleftarrow{\mathrm{nm}}_d$ and $\widetilde{\mathrm{nm}}_d$, and then prove Theorem \ref{theorem-inverse-limit-compactify}.  Finally we introduce the Deligne-Mumford compactifications $\widehat{\mathrm{nm}}^\ast_d$ in section \ref{Deligne-Mumford} and prove Theorem \ref{theorem-Deligne-Mumford-GIT} in section \ref{DM to GIT}.

\section{Background}\label{background}
\subsection{Stable and Semistable Loci}
In this subsection, we describe the stable and semistable loci in $\mathbb{P}^{2d+1}$, which were computed by Silverman \cite{Silverman98}. Equivalent conditions were given by DeMarco \cite{DeMarco07}. \par
Denote $\mathrm{Rat}_d^s$ to be the stable locus in $\mathbb{P}^{2d+1}$ and denote $\mathrm{Rat}_d^{ss}$ to be the semistable locus in $\mathbb{P}^{2d+1}$. Then they are algebraic sets containing $\mathrm{Rat}_d$. Moreover, 

\begin{proposition}\cite[Proposition 2.2]{Silverman98}\label{seimstable}
Let $f\in\mathbb{P}^{2d+1}$. Then 
 $f\not\in\mathrm{Rat}_d^{ss}$ if and only if there exists $M\in \mathrm{PSL}_2(\mathbb{C})$ such that 
$$M^{-1}\circ f\circ M=[a_d:\cdots:a_0:b_d:\cdots:b_0]$$
 with $a_i=0$ for all $i\ge(d+1)/2$ and $b_i=0$ for all $i\ge(d-1)/2$.\par
Similarly, $f\not\in\mathrm{Rat}_d^{s}$ if and only if there exists $M\in \mathrm{PSL}_2(\mathbb{C})$ such that 
$$M^{-1}\circ f\circ M=[a_d:\cdots:a_0:b_d:\cdots:b_0]$$
with $a_i=0$ for all $i>(d+1)/2$ and $b_i=0$ for all $i>(d-1)/2$.
\end{proposition}
The action of $\mathrm{PSL}_2(\mathbb{C})$ by conjugation on $\mathrm{Rat}_d$ extends to $\mathbb{P}^{2d+1}$. The stable locus $\mathrm{Rat}^s_d$ is the largest open $\mathrm{PSL}_2(\mathbb{C})$-invariant subset of $\mathbb{P}^{2d+1}$, so $\mathrm{rat}_d^s:=\mathrm{Rat}_d^s/\mathrm{PSL}_2(\mathbb{C})$ is Hausdorff. We say $f,g\in\mathrm{Rat}_d^{ss}$ are GIT conjuage if the Zariski closures of their $\mathrm{PSL}_2(\mathbb{C})-$orbits have a point in common, i.e. 
$$\overline{\{M^{-1}\circ f\circ M: M\in\mathrm{PSL}_2(\mathbb{C})\}}\cap\overline{\{M^{-1}\circ g\circ M: M\in\mathrm{PSL}_2(\mathbb{C})\}}\not=\emptyset.$$ 
Let $\mathrm{rat}_d^{ss}:=\mathrm{Rat}_d^{ss}//\mathrm{PSL}_2(\mathbb{C})$ be the space of GIT conjugate classes. Then $\mathrm{rat}_d^{ss}$ is compact. Moreover, $\mathrm{rat}_d^s\subset\mathrm{rat}_d^{ss}$ since the GIT conjugation coincides to the conjugation on the set $\mathrm{Rat}_d^s$. We call $\overline{\mathrm{rat}}_d:=\mathrm{rat}_d^{ss}$ the GIT compactification of $\mathrm{rat}_d$, and denote $[f]_{\mathrm{GIT}}\in\overline{\mathrm{rat}}_d$ to be the GIT conjugate class of $f\in\mathrm{Rat}_d^{ss}$. Note $\mathrm{Rat}_d^s=\mathrm{Rat}_d^{ss}$ if and only if $d$ is even. Hence $\overline{\mathrm{rat}}_d$ is a geometric quotient for $d$ even and a categorical quotient for $d$ odd, and it is proper quotient scheme over $\mathbb{Z}$ \cite[Theorem 2.1]{Silverman98}.\par
Recall that for $f\in\mathbb{P}^{2d+1}$, there exist two degree $d$ homogeneous polynomials $F_a, F_b\in\mathbb{C}[X,Y]$ such that 
$$f([X:Y])=[F_a(X,Y):F_b(X,Y)]=H_f(X,Y)\hat f([X:Y]),$$
where $H_f=\gcd(F_a,F_b)$ and $\hat f$ is a rational map of degree at most $d$. We say a zero of $H_f$ is a hole of $f$, and for a point $z\in\mathbb{P}^1$,we say the multiplicity of $z$ as a zero of $H_f$ is the depth of $f$ at $z$ denoted by $d_z(f)$.\par
Considering the holes and depths, DeMarco \cite{DeMarco07} gave another way to describe the sets $\mathrm{Rat}_d^s$ and $\mathrm{Rat}_d^{ss}$.
\begin{lemma}\cite[Section 3]{DeMarco07}\label{stable-hole}
Let $f=H_f\hat f\in\mathbb{P}^{2d+1}$.
\begin{enumerate}
\item If $d\ge 2$ is even, $f\in\mathrm{Rat}_d^{s}$ if and only if the depth of each hole is $\le d/2$ and if $d_h(f)=d/2$, then $\hat f(h)\not=h$.
\item If $d\ge 3$ is odd, $f=H_f\hat f\in\mathrm{Rat}_d^{ss}$ if and only if the depth of each hole is $\le(d+1)/2$ and if $d_h(f)=(d+1)/2$, then $\hat f(h)\not=h$.
\item If $d\ge 3$ is odd, $f\in\mathrm{Rat}_d^{s}$ if and only if the depth of each hole is $\le(d-1)/2$ and if $d_h(f)=(d-1)/2$, then $\hat f(h)\not=h$.
\end{enumerate}
\end{lemma}

\subsection{The Space $\mathrm{NM}_d$ of Newton Maps}
In this subsection, we describe the (semi)stable degenerate Newton maps. Let $\overline{\mathrm{NM}}_d$ be the closure of $\mathrm{NM}_d$ in $\mathbb{P}^{2d+1}$. Based on Lemma \ref{stable-hole}, for Newton maps, we have 
\begin{lemma}\label{Newton-hole}
For $d\ge 2$, let $N\in\overline{\mathrm{NM}}_d$.
\begin{enumerate}
\item If $d$ is even, $N\in\overline{\mathrm{NM}}_d\cap\mathrm{Rat}_d^{s}$ if and only if  the depth of each hole of $N$ is $\le d/2-1$.
\item If $d$ is odd, $N\in\overline{\mathrm{NM}}_d\cap\mathrm{Rat}_d^{ss}$ if and only if the depth of each hole of $N$ is $\le(d-1)/2$.
\item If $d$ is odd, $N\in\overline{\mathrm{NM}}_d\cap\mathrm{Rat}_d^{s}$ if and only if the depth of each hole of $N$ is $\le(d-3)/2$.
\end{enumerate}
\end{lemma}
\begin{proof}
Write $N=H_N\widehat N$. Note each hole, if exists,  of $N$ is a fixed point of $\widehat N$. Then the conclusions follow from Lemma \ref{stable-hole}.
\end{proof}
\begin{corollary}\label{cubic-stable-empty}
For $d\ge 2$, $\partial\overline{\mathrm{NM}}_d\cap\mathrm{Rat}_d^{s}=\emptyset$ if and only if $d=3$.
\end{corollary}
We give the following example to figure out $\overline{\mathrm{NM}}_3\cap \mathrm{Rat}_d^{ss}$ precisely.
\begin{example}\label{cubic Newton}
Cubic Newton maps.\par
If $N=H_N\widehat N\in\partial\overline{\mathrm{NM}}_3\cap\mathrm{Rat}_3^{ss}$, then the depth of each hole of $N$ is $1$. We have three cases.\par 
Case I: $N$ has exact one hole and it is $\infty$.\par 
In this case, there exist two distinct points $r_1$ and $r_2$ in $\mathbb{C}$ such that $N=N_{\{r_1,r_2,\infty\}}$. Thus 
$$N([X:Y])=YN_{\{r_1,r_2\}}([X:Y]).$$\par 
Case II: $N$ has exact one hole and it is $r_1\in\mathbb{C}$.\par 
In this case, there exists a point $r_2\in\mathbb{C}\setminus\{r_1\}$ such that $N=N_{\{r_1,r_1,r_2\}}$. Thus 
$$N([X:Y])N=(X-r_1Y)\widehat N_{\{r_1,r_1,r_2\}}([X:Y]).$$\par
Case III: $N$ has two distinct holes. 
In this case, $\infty$ is a hole of $N$. Then there exists a point $r_1\in\mathbb{C}$ such that $N=N_{\{r_1,r_1,\infty\}}$.
$$N([X:Y])=(X-r_1Y)Y\widehat N_{\{r_1,r_1\}}([X:Y]).$$\par 
Thus in any case, $N\not\in\mathrm{Rat}_3^{s}$.
\end{example}
Recall that $I(d)\subset\mathbb{P}^{2d+1}$ is the indeterminacy locus, that is, 
$$I(d)=\{f=H_f\hat f\in\mathbb{P}^{2d+1}:\hat f\equiv c\in\mathbb{P}^1, H_f(c)=0\}.$$
The following result states that there is no semistable indeterminacy points in the space $\overline{\mathrm{NM}}_d$.
\begin{proposition}\label{Newton-Seim-Indeter}
For any $d\ge 2$, 
$$\overline{\mathrm{NM}}_d\cap\mathrm{Rat}_d^{ss}\cap I(d)=\emptyset.$$ 
In particular, $\partial\overline{\mathrm{NM}}_d\cap\mathrm{Rat}_d^{ss}=\emptyset$ if and only if $d=2$.
\end{proposition}
\begin{proof}
By Lemma \ref{Newton-hole}, if $N=H_N\widehat N\in\overline{\mathrm{NM}}_d\cap\mathrm{Rat}_d^{ss}$, then $\deg\widehat N\ge 1$. So $N\not\in I(d)$.\par
If $d=2$, let $N=H_N\widehat N\in\partial\overline{\mathrm{NM}}_2$, then $\deg H_N=1$. Note $\widehat N$ fixes the hole of $N$. By Lemma \ref{Newton-hole}, we have $N\not\in\mathrm{Rat}_2^{ss}$.\par
If $d\ge 3$, then there exists $N=H_N\widehat N\in\partial\overline{\mathrm{NM}}_d$ such that $\widehat N\in\mathrm{NM}_{d-1}$. Thus $\deg H_N=1\le(d-1)/2$. By Lemma \ref{Newton-hole}, we have $N\in\mathrm{Rat}_d^{ss}$
\end{proof}

\subsection{Measures of (Degenerate) Rational Maps}
In this section, following DeMarco \cite{DeMarco05}, we associate each $f\in\mathbb{P}^{2d+1}$ a probability measure $\mu_f$.\par
For $f\in\mathrm{Rat}_d$, let $\mu_f$ be the unique measure of maximal entropy, which is given by the weak limit 
$$\mu_f=\lim\limits_{n\to\infty}\frac{1}{d^n}\sum_{f^n(z)=a}\delta_z$$
for any nonexceptional point $a\in\mathbb{P}^1$, see \cite{Freire83, Ljubich83, Mane83}. The measure $\mu_f$ has no atoms, and $\text{supp}\ \mu_f=J(f)$. Moreover, $\mu_{f^n}=\mu_f$.\par 
For $f=H_f\hat f$ with $\text{deg}\ \hat f\ge 1$, define 
$$\mu_f=\sum_{n=0}^\infty\frac{1}{d^{n+1}}\sum_{H_f(h)=0}\sum_{\hat f^n(z)=h}\delta_z,$$
where the holes $h$ and all preimages by $\hat f$ are counted with multiplicity. Then $\mu_f$ is an atomic probability measure. If $\text{deg}\hat f=0$, define 
$$\mu_f=\frac{1}{d}\sum_{H_f(h)=0}\delta_h,$$
where the holes $h$ are counted with multiplicity. 
Then if $f\not\in I(d)$, we have $\mu_f=\mu_{f^n}$.\par
The following result claims that away from the indeterminacy locus $I(d)$, the map, sending $f\in\mathbb{P}^{2d+1}$ to $\mu_f$, is good behaved.
\begin{proposition}\label{measure-converge}\cite[Theorem 0.2]{DeMarco05}
Suppose that $d\ge2$ and $f\in\mathbb{P}^{2d+1}$. Then $f\not\in I(d)$ if and only if the map $g\to\mu_g$ is continuous at $f$. 
\end{proposition}

\section{GIT Compactifications}\label{GIT}

\subsection{GIT Compactifications for $\mathrm{nm}_d$ }
Recall $\mathrm{nm}_d$ is the moduli space of the degree $d\ge 2$ Newton maps. In this subsection, we study the GIT compactification for $\mathrm{nm}_d$.\par 
Note
$$\mathrm{nm}_d\subset\mathrm{rat}_d\subset\overline{\mathrm{rat}}_d.$$ 
Let $\overline{\mathrm{nm}}_d\subset\overline{\mathrm{rat}}_d$ be the closure of $\mathrm{nm}_d$ in $\overline{\mathrm{rat}}_d$. Then $\overline{\mathrm{nm}}_d$ is compact and contains $\mathrm{nm}_d$ as a dense subset. We say $\overline{\mathrm{nm}}_d$ is the GIT compactification of $\mathrm{nm}_d$.
\begin{lemma}
For $d\ge 2$,
$$\overline{\mathrm{nm}}_d=\{[N]_{\mathrm{GIT}}: N\in \overline{\mathrm{NM}}_d\cap\mathrm{Rat}_d^{ss}\}.$$
\end{lemma}
\begin{proof}
First note $\{[N]_{\mathrm{GIT}}: N\in \overline{\mathrm{NM}}_d\cap\mathrm{Rat}_d^{ss}\}\subset\overline{\mathrm{rat}}_d$ is compact and contains $\mathrm{nm}_d$. Thus, $\overline{\mathrm{nm}}_d\subset\{[N]_{\mathrm{GIT}}: N\in \overline{\mathrm{NM}}_d\cap\mathrm{Rat}_d^{ss}\}$.\par
For other direction, for $f\in\overline{\mathrm{NM}}_d\cap\mathrm{Rat}_d^{ss}$, let $\{f_n\}\subset\mathrm{NM}_d$ such that $f_n\xrightarrow{sa} f$ as $n\to\infty$. Then $[f_n]_{\mathrm{GIT}}$ converges to $[f]_{\mathrm{GIT}}$ in $\overline{\mathrm{rat}}_d$. So $[f]_{\mathrm{GIT}}\in\overline{\mathrm{nm}}_d$. Hence $\{[N]_{\mathrm{GIT}}: N\in \overline{\mathrm{NM}}_d\cap\mathrm{Rat}_d^{ss}\}\subset\overline{\mathrm{nm}}_d$.
\end{proof}
To describe the space $\overline{\mathrm{nm}}_d$, we first show that all the strictly semistable points in $\overline{\mathrm{NM}}_d$ have the same image in $\overline{\mathrm{nm}}_d$.
\begin{proposition}\label{semi-singleton}
If $d\ge 3$ is odd, then the set $\{[N]_{\mathrm{GIT}}: N\in \overline{\mathrm{NM}}_d\cap\mathrm{Rat}_d^{ss}\setminus\mathrm{Rat}_d^s\}$ is the singleton $\{[X^{(d-1)/2}Y^{(d-1)/2}[(d-1)X:(d+1)Y]]_{\mathrm{GIT}}\}$.
\end{proposition}
\begin{proof}
For any $N=H_N\widehat N\in\overline{\mathrm{NM}}_d\cap(\mathrm{Rat}_d^{ss}\setminus \mathrm{Rat}_d^s)$, by Lemma \ref{Newton-hole}, we know that $N$ has a hole with depth $(d-1)/2$. Thus, there are three cases.\par
Case I: $N$ has only one hole with depth $(d-1)/2$ and it is $z=\infty$. Then $\widehat N$ is a Newton map for a polynomial 
$$P(z)=(z-r_1)^{m_1}\cdots(z-r_n)^{m_n},$$ 
where $r_1,\cdots,r_n$ are $n$ distinct points in $\mathbb{C}$ and $m_1+\cdots+m_n=(d+1)/2$. Then 
$$N=H_N\widehat N=(X-r_1Y)^{m_1-1}\cdots(X-r_nY)^{m_n-1}Y^{\frac{d-1}{2}}[\widehat N_a:\hat N_b],$$
where 
$$[\widehat N_a:\widehat N_b]=\Bigg[X\sum_{i=1}^nm_i\prod_{\substack{j=1\\ j\not=i}}^n(X-r_jY)-\prod_{i=1}^n(X-r_jY):Y\sum_{i=1}^nm_i\prod_{\substack{j=1\\ j\not=i}}^n(X-r_jY)\Bigg].$$
Let $M_t([X:Y])=[X/t:tY]$. Then as $t\to 0$,
\begin{align*}
&M_t^{-1}\circ N\circ M_t\\
&=[tH_N(X/t,tY)\widehat N_a(X/t,tY):(1/t)H_N(X/t,tY)\widehat N_b(X/t,tY)]\\
&=(X-r_1t^2Y)^{m_1-1}\cdots(X-r_nt^2Y)^{m_n-1}Y^{\frac{d-1}{2}}[t^2\widehat N_a(X,t^2Y):\widehat N_b(X,t^2Y)]\\
&\to X^{\sum_{i=1}^nm_i-n}Y^{\frac{d-1}{2}}[(1-1/\sum_{i=1}^nm_i)X^n:X^{n-1}Y]\\
&=X^{\frac{d-1}{2}}Y^{\frac{d-1}{2}}[(d-1)X:(d+1)Y].
\end{align*}
Case II: $N$ has only one hole with depth $(d-1)/2$ and it is $z=a\in\mathbb{C}$. By conjugating, we can assume $a=0$. Then $\widehat N$ is a Newton map for a polynomial 
$$P(z)=z^{\frac{d+1}{2}}(z-r_1)^{m_1}\cdots(z-r_n)^{m_n},$$ 
where $r_1,\cdots,r_n$ are $n$ distinct points in $\mathbb{C}\setminus\{0\}$ and $m_1+\cdots+m_n+d_\infty(N)=(d-1)/2$. Then 
$$N=H_N\widehat N=X^{\frac{d-1}{2}}(X-r_1Y)^{m_1-1}\cdots(X-r_nY)^{m_n-1}Y^{d_\infty(N)}[\widehat N_a:\widehat N_b],$$
where 
\begin{align*}
[\widehat N_a:\widehat N_b]=&\Bigg[X\Big(\frac{d-1}{2}\prod_{i=1}^n(X-r_iY)+\sum_{i=1}^nm_i\prod_{\substack{j=1\\ j\not=i}}^n(X-r_iY)\Big):\\
&\Big(\frac{d+1}{2}\prod_{i=1}^n(X-r_iY)+\sum_{i=1}^nm_i\prod_{\substack{j=1\\ j\not=i}}^n(X-r_iY)\Big)Y\Bigg].
\end{align*}
Let $M_t([X:Y])=[tX:Y/t]$. Then as $t\to 0$,
\begin{align*}
&M_t^{-1}\circ N\circ M_t\\
&=[(1/t)H_f(tX,Y/t)\widehat N_a(tX,Y/t):tH_f(tX,Y/t)\hat N_b(tX,Y/t))]\\
&=X^{\frac{d-1}{2}}(t^2X-r_1Y)^{m_1-1}\cdots(t^2X-r_nY)^{m_n-1}Y^{d_\infty(N)}[\widehat N_a(t^2X,Y):t^2\widehat N_b(t^2X,Y)]\\
&\to X^{\frac{d-1}{2}}Y^{\sum_{i=1}^nm_i-n}Y^{d_\infty(N)}[(1-2/(d+1))XY^n:Y^{n+1}]\\
&=X^{\frac{d-1}{2}}Y^{\frac{d-1}{2}}[(d-1)X:(d+1)Y].
\end{align*}
Case III: $N$ has two holes and both of them have depth $(d-1)/2$. Then one of these two holes must be $\infty$. Indeed, if both holes were finite, then there would be $(d+1)/2$ many roots collided at each hole. It is impossible. By conjugating an affine map, we can assume the other is $z=0$. Then 
$$N=H_N\widehat N=X^{\frac{d-1}{2}}Y^{\frac{d-1}{2}}[(d-1)X:(d+1)Y].$$\par 
Note 
$$X^{\frac{d-1}{2}}Y^{\frac{d-1}{2}}[(d-1)X:(d+1)Y]\in\mathrm{Rat}_d^{ss}.$$ 
The conclusion follows from the GIT conjugacy.
\end{proof}
From Example \ref{cubic Newton}, we know there are three conjugacy classes in $\partial\overline{\mathrm{NM}}_3\cap\mathrm{Rat}_3^{ss}$. The following corollary claims all of them have same image under GIT quotient.
\begin{corollary}
For cubic Newton's method, 
$$\{[N]_{\mathrm{GIT}}: N\in\partial\overline{\mathrm{NM}}_3\cap\mathrm{Rat}_3^{ss}\}=[XY[X:2Y]]_{\mathrm{GIT}}.$$
\end{corollary}
\begin{proof}
By Corollary \ref{cubic-stable-empty}, 
$$\partial\overline{\mathrm{NM}}_3\cap\mathrm{Rat}_3^{ss}=\partial\overline{\mathrm{NM}}_3\cap(\mathrm{Rat}_3^{ss}\setminus\mathrm{Rat}_3^s).$$
Then the conclusion follows from Proposition \ref{semi-singleton}.
\end{proof}
Proposition \ref{semi-singleton} shows $\overline{\mathrm{nm}}_d$ is the one point compactification of the space $\{[N]_{\mathrm{GIT}}: N\in\overline{\mathrm{NM}}_d\cap\mathrm{Rat}_d^s\}$. So for $d=3$, we have 
\begin{corollary}
The compactification $\overline{\mathrm{nm}}_3$ is homeomorphic to $\mathbb{P}^1$.
\end{corollary}
Now we consider the set $\{[N]_{\mathrm{GIT}}: N\in\overline{\mathrm{NM}}_d\cap\mathrm{Rat}_d^s\}$. We need the following results for the points in $\overline{\mathrm{NM}}_d$. Let $f=H_f\hat f$ and $g=H_g\hat g$ in $\overline{\mathrm{NM}}_d$ such that $\deg\hat f\ge 1$ and $\deg\hat g\ge 1$. Then a conjugacy between $f$ and $g$ induces a conjugation between $\hat f$ and $\hat g$. Conversely, the following lemma claims that a conjugacy between $\hat f$ and $\hat g$ also induces a conjugation between $f$ and $g$.  
\begin{lemma}\label{Newton-Conj-Hole}
Fix $d\ge 2$. Suppose $f=H_f\hat f, g=H_g\hat g\in\overline{\mathrm{NM}}_d$ such that $\deg\hat f\ge 1$ and $\deg\hat g\ge 1$. If $M\in \mathrm{PSL}_2(\mathbb{C})$ is such that $\hat f=M^{-1}\circ\hat g\circ M$, then the holes of $g$ are the images under $M$ of the holes of $f$ and $d_a(f)=d_{M(a)}(g)$ for any $a\in\widehat{\mathbb{C}}$. In particular, we have $f=M^{-1}\circ g\circ M$.
\end{lemma}
\begin{proof}
Note $\hat f$ is a Newton map for some polynomial $P_1$ of degree at most $d$. Let $r_1,\cdots,r_n$ be the distinct zeros of $P_1$. Then we can write $P_1(z)=\prod_{i=1}^n(z-r_i)^{\alpha_i}$ for some $\alpha_i\in\mathbb{Z}$. Thus 
$$\hat f(z)=z-\frac{1}{\frac{\alpha_1}{z-r_1}+\cdots+\frac{\alpha_n}{z-r_n}}$$
Similarly, there exist $m$ distinct points $s_1,\cdots,s_m$ in $\mathbb{C}$ and $m$ integers (not necessary distinct) $\beta_1,\cdots,\beta_m$ such that 
$$\hat g(z)=z-\frac{1}{\frac{\beta_1}{z-s_1}+\cdots+\frac{\beta_m}{z-s_m}}.$$
Thus, we have $d_{r_i}(f)=\alpha_i-1$ for $i=1,\cdots, n$ and $d_{s_j}(g)=\beta_j-1$ for $j=1,\cdots, m$. Since $\hat f=M^{-1}\circ\hat g\circ M$, then $\deg\hat f=\deg\hat g$. Hence $n=m$. Note $r_1,\cdots, r_n$ are the only (super)attracting fixed points of $\hat f$ and $s_1,\cdots, s_m$ are the only (super)attracting fixed points of $\hat g$. Thus $M$ maps $\{r_1,\cdots,r_n\}$ to $\{s_1,\cdots,s_m\}$. Moreover, the multiplier of $\hat f$ at $r_i$ equals to the multiplier of $\hat g$ at $M(r_i)$. It immediately implies that $d_{r_i}(f)=d_{M(r_i)}(g)$.
\end{proof}
\begin{corollary}
For $d\ge 2$,
$$\overline{\mathrm{nm}}_d=\overline{\mathrm{NM}}_d\cap\mathrm{Rat}_d^{ss}//\mathrm{Aut(\mathbb{C})}.$$
\end{corollary} 
\begin{proof}
First note if $N=H_N\widehat N\in\overline{\mathrm{NM}}_d\cap\mathrm{Rat}_d^{s}$ then $\deg\widehat N\ge 2$. Indeed, since $N\in\mathrm{NM}_d\cap\mathrm{Rat}_d^{s}$, by Lemma \ref{Newton-hole}, we know for any $z\in\mathbb{P}^1$, the depth $d_z(N)<(d-1)/2$. If $\deg\widehat N\le 1$, then $N$ has at most on hole in $\mathbb{C}$. It is impossible. Thus $\widehat N$ has unique repelling fixed points at $z=\infty$. Note for $N_1=H_{N_1}\widehat N_1, N_2=H_{N_2}\widehat N_2\in\overline{\mathrm{NM}}_d\cap\mathrm{Rat}_d^{s}$ with $[N_1]_{\mathrm{GIT}}=[N_2]_{\mathrm{GIT}}$, we have $\widehat N_1$ is conjugate to $\widehat N_2$. Thus there exists $M\in\mathrm{Aut}(\mathbb{C})$ such that $M^{-1}\circ\widehat N_1\circ M=\widehat N_2$. Thus by Lemma \ref{Newton-Conj-Hole}. $M^{-1}\circ N_1\circ M=N_2$.\par
For $N\in\overline{\mathrm{NM}}_d\cap\mathrm{Rat}_d^{ss}\setminus\mathrm{Rat}_d^s$, it follows Proposition \ref{semi-singleton}.
\end{proof}
\begin{corollary}\label{Newton-Conj-Hole-Affine}
Suppose $f=H_f\hat f, g=H_g\hat g\in\overline{\mathrm{NM}}_d$ such that $deg\hat f\ge 1$ and $\deg\hat g\ge 1$. Suppose $M\in\mathrm{PSL}_2(\mathbb{C})$ such that $f=M^{-1}\circ g\circ M$. Then $M\in \mathrm{Aut}(\mathbb{C})$.
\end{corollary}
\begin{proof}
Note $z=\infty$ is the unique repelling fixed point of $\hat f$ and $\hat g$. Hence, $M$ fixes $\infty$. So $M$ is affine.
\end{proof}
However, the converse of Lemma \ref{Newton-Conj-Hole} is not true. For example, let $f=H_f\hat f,g=H_g\hat g\in\overline{\mathrm{NM}}_6$ be the degenerate Newton maps corresponding to the two polynomials $P_1(z)=z^3(z-1)^2(z-2)$ and $P_2(z)=z^3(z-1)^2(z-3)$, respectively. Then $f$ and $g$ have same holes and same corresponding depths. Indeed, $\mathrm{Hole}(f)=\mathrm{Hole}(g)=\{0,1\}$ with $d_f(0)=d_g(0)=2$ and $d_f(1)=d_g(1)=1$. But $\hat f$ is not conjugate to $\hat g$. Indeed, if there had $M\in\mathrm{PSL}_2(\mathbb{C})$ such that $\hat f=M^{-1}\circ\hat g\circ M$, then by Lemma \ref{Newton-Conj-Hole} and Corollary \ref{Newton-Conj-Hole-Affine}, we know $M(z)=z$. Note $\hat f\not=\hat g$. Thus, $\hat f$ is not conjugate to $\hat g$. However, if we consider all the fixed points of $\hat f$ and $\hat g$ in $\mathbb{C}$, we have 
\begin{lemma}\label{Newton-Conj-Root}
Suppose $f=H_f\hat f, g=H_g\hat g\in\overline{\mathrm{NM}}_d$ such that $\deg\hat f=\deg\hat g\ge 1$. Let $M\in\mathrm{PSL}_2(\mathbb{C})$ send the (super)attracting fixed points of $\hat f$ to the (super)attracting fixed points of $\hat f$ and $d_a(f)=d_{M(a)}(g)$ if $a$ is a (super)attracting fixed points of $\hat f$, then $\hat f=M^{-1}\circ\hat g\circ M$.
\end{lemma}
\begin{proof}
For a (degenerate) Newton map $N=H_N\widehat N\in\overline{\mathrm{NM}}_d$ with $\deg\widehat{N}\ge 1$, the depths of $N$ at holes determine the multipliers of the (super)attracting fixed points for the map $\widehat{N}$. Moreover, the locations and the corresponding multipliers of the (super)attracting fixed points for the map $\widehat{N}$ determine the (degenerate) Newton map $N$. Thus $\hat f=M^{-1}\circ\hat g\circ M$.
\end{proof}
Combining with Lemmas \ref{Newton-Conj-Hole} and \ref{Newton-Conj-Root}, we have 
\begin{proposition}\label{Newton-Conj-fixed-pt-depth}
Suppose $f=H_f\hat f, g=H_g\hat g\in\overline{\mathrm{NM}}_d$ such that $\deg\hat f\ge 1$ and $\deg\hat g\ge 1$. Then $f$ is conjugate to $g$ if and only if there exits $M\in\mathrm{Aut}(\mathbb{C})$ such that $M$ maps the (super)attracting fixed points of $\hat f$ to the (super)attracting fixed points of $\hat g$ and satisfies $d_a(f)=d_{M(a)}(g)$ for all (super)attracting fixed points $a$ of $\hat f$.
\end{proposition}
For a (degenerate) Newton map $N$, we associate each point $z\in\mathbb{C}$ a nonnegative integer $m(N,z)$. If $z$ is a hole of $N$, set $m(N, z)=d_z(N)+1$. If $z$ is an superattracting fixed point of $N$, set $m(N, z)=1$. Otherwise, set $m(N, z)=0$. Note the depth $d_\infty(N)$ is determined by $\deg\widehat N$ and the depths of the finite holes. Then we can define a divisor for $N$ on $\mathbb{C}$ by $D(N)=\sum_{z\in\mathbb{C}}m(N, z)z$. By Proposition \ref{Newton-Conj-fixed-pt-depth}, to character the space $\overline{\mathrm{NM}}_d\cap \mathrm{Rat}_d^{s}/\mathrm{Aut}(\mathbb{C})$, we need to study the divisors $D(N)$.
\begin{proposition}\label{odd-Newton-stable-moduli}
For odd $d\ge 3$ and $(d+3)/2\le n\le d$, let 
$$D^n=\{(m_1,\cdots,m_n)\in\mathbb{N}^n: 0\le m_i-1<\frac{d-1}{2}\}\subset\mathbb{N}^n,$$
then 
$$\overline{\mathrm{NM}}_d\cap\mathrm{Rat}_d^{s}/\mathrm{Aut}(\mathbb{C})\cong\bigcup_{n=(d+3)/2}^{d}(((\mathbb{C}^n\setminus\Delta_n\times D^n)/\sim)/\mathrm{Aut}(\mathbb{C})\times id),$$
where $((r_1,\cdots,r_n),(m_1,\cdots,m_n))\sim((\tilde r_1,\cdots,\tilde r_n),(\widetilde m_1,\cdots,\widetilde m_n))$ if there exists $\sigma\in\mathrm{Sym}(n)$ such that $\sigma(r_1,\cdots,r_n)=(\tilde r_1,\cdots,\tilde r_n)$ and $\sigma(m_1,\cdots,m_n)=(\widetilde m_1,\cdots,\widetilde m_n)$.
\end{proposition}
\begin{proof}Define 
$$\widetilde{F}:\overline{\mathrm{NM}}_d\cap\mathrm{Rat}_d^{s}\to\bigcup_{n=(d+3)/2}^{d}(((\mathbb{C}^n\setminus\Delta_n\times D^n)/\sim)$$
sending $f=H_f\hat f$ to $((r_1,\cdots,r_n),(m_1,\cdots,m_n))$, where 
$$\hat f(z)=z-\frac{1}{\frac{m_1}{z-r_1}+\cdots+\frac{m_n}{z-r_n}}.$$
Then $\widetilde{F}$ induces a map $F$ such that the following diagram commutes,
\[
\begin{CD}
\overline{\mathrm{NM}}_d\cap\mathrm{Rat}_d^{s} @>\widetilde F>>\bigcup\limits_{n=(d+3)/2}^{d}((\mathbb{C}^n\setminus\Delta_n\times D^n)/\sim)\\
@V/\mathrm{Aut}(\mathbb{C})VV@V/\mathrm{Aut}(\mathbb{C})\times id VV\\
\overline{\mathrm{NM}}_d\cap\mathrm{Rat}_d^{s}/\mathrm{Aut}(\mathbb{C})@>F>>\bigcup\limits_{n=(d+3)/2}^{d}(((\mathbb{C}^n\setminus\Delta_n\times D^n)/\sim)/\mathrm{Aut}(\mathbb{C})\times id).
\end{CD}
\]
Indeed, $F$ is well-defined. In fact, by Lemma \ref{Newton-Conj-Hole}, if $[H_g\hat g]=[H_f\hat f]$, then $\hat g$ is conjugate to $\hat f$. It is easy to check $F([H_g\hat g])=F[H_f\hat f]$. Furthermore, by Lemma \ref{Newton-Conj-Hole} and Lemma \ref{Newton-Conj-Root}, $F$ is a bijection.\par
It is easy to show $\tilde F$ is a homeomorphism. So $F$ is a homeomorphism.
\end{proof}
Similarly, for even degrees, we have 
\begin{proposition}\label{even-Newton-stable-moduli}
For even $d\ge 2$ and $d/2+2\le n\le d$, let 
$$D^n=\{(m_1,\cdots,m_n)\in\mathbb{N}^n: 0\le m_i-1<d/2-1\}\subset\mathbb{N}^n,$$
then 
$$\overline{\mathrm{NM}}_d\cap\mathrm{Rat}_d^{s}/\mathrm{Aut}(\mathbb{C})\cong\bigcup_{n=(d+4)/2}^{d}(((\mathbb{C}^n\setminus\Delta_n\times D^n)/\sim)/\mathrm{Aut}(\mathbb{C})\times id),$$
where $((r_1,\cdots,r_n),(m_1,\cdots,m_n))\sim((\tilde r_1,\cdots,\tilde r_n),(\widetilde m_1,\cdots,\widetilde m_n))$ if there exists $\sigma\in\mathrm{Sym}(n)$ such that $\sigma(r_1,\cdots,r_n)=(\tilde r_1,\cdots,\tilde r_n)$ and $\sigma(m_1,\cdots,m_n)=(\widetilde m_1,\cdots,\widetilde m_n)$.
\end{proposition}
Now we give an example to illustrate the set $\partial\overline{\mathrm{NM}}_d\cap\mathrm{Rat}_d^{ss}//\mathrm{Aut}(\mathbb{C})$.
\begin{example}
Quartic Newton maps.\par
Note $\mathrm{Rat}_4^{ss}=\mathrm{Rat}_4^{s}$. Then 
$$\partial\overline{\mathrm{NM}}_4\cap\mathrm{Rat}_4^{ss}//\mathrm{Aut}(\mathbb{C})=\partial\overline{\mathrm{NM}}_4\cap\mathrm{Rat}_4^{s}/\mathrm{Aut}(\mathbb{C}).$$
By Proposition \ref{even-Newton-stable-moduli}, we have
\begin{align*}
\partial\mathrm{nm}_4&=\partial\overline{\mathrm{NM}}_4\cap\mathrm{Rat}_4^{ss}//\mathrm{Aut}(\mathbb{C})\\
&=\{[N_{\{0,0,1,\infty\}}],[N_{\{0,0,1,1\}}],[N_{\{0,1,c,\infty\}}],[N_{\{0,0,1,c\}}]:c\in\mathbb{C}\setminus\{0,1\}\}.
\end{align*}
\end{example}
For $f\in\overline{\mathrm{NM}}_d\cap\mathrm{Rat}_d^{s}$, any representative $g$ for $[f]_{\mathrm{GIT}}$ is conjugate to $f$. However, this is not the case for $f\in\overline{\mathrm{NM}}_d\cap(\mathrm{Rat}_d^{ss}\setminus\mathrm{Rat}_d^s)$.
\begin{example} 
Among cubic Newton maps, let 
$$N([X:Y])=H_N(X,Y)\widehat N([X:Y])=N_{\{0,0,\infty\}}([X:Y])=XY[X:2Y]$$ 
and 
$$f([X:Y])=H_f(X,Y)\hat f([X:Y])=XY[X:2X+2Y].$$ 
Let $M_t([X:Y])=[tX:Y]$. Then we have, as $t\to 0$,
$$M_t^{-1}\circ f\circ M_t\to N.$$
Thus, $[f]_{\mathrm{GIT}}=[N]_{\mathrm{GIT}}\in\overline{\mathrm{rat}}_3$. Note $[N]_{\mathrm{GIT}}\in\overline{\mathrm{nm}}_3$. Hence $[f]_{\mathrm{GIT}}\in\overline{\mathrm{nm}}_3$. However, $f$ is not conjugate to any element in $\overline{\mathrm{NM}}_3\cap\mathrm{Rat}_3^{ss}$. Indeed, if $f$ was conjugate to some element in $\overline{\mathrm{NM}}_3\cap\mathrm{Rat}_3^{ss}$, then $f$ would be conjugate to $N$ since $\deg\hat f=1$ and up to conjugacy, $N$ is the unique element in $\overline{\mathrm{NM}}_3\cap\mathrm{Rat}_3^{ss}$ with two holes. Then there would exist $M\in\mathrm{PSL}_2(\mathbb{C})$ such that $M^{-1}\circ f\circ M=N$. So $M$ maps $\{0,\infty\}$ to $\{0,\infty\}$. Hence $M(z)$ has form $M(z)=kz$ or $M(z)=k/z$ for some $k\in\mathbb{C}\setminus\{0\}$. However, there is no $k\in\mathbb{C}\setminus\{0\}$ such that $M^{-1}\circ\hat f\circ M=\hat N$. It is a contradiction.
\end{example} 

\subsection{Iterate Maps on $\overline{\mathrm{NM}}_d$}
In this subsection, we show the iterate maps 
$$\Psi_n:\mathbb{P}^{2d+1}\dasharrow\mathbb{P}^{2d^n+1},$$ 
sending $f$ to $f^n$, preserve the (semi)stability of points in $\overline{\mathrm{NM}}_d$.\par  
Recall that for any $n\ge 2$, indeterminacy loci of $\Psi_n$ are $I(d)$ \cite[Lemma 2.2]{DeMarco05}. Then by Proposition \ref{Newton-Seim-Indeter}, the maps $\Psi_n$ are well-defined on $\overline{\mathrm{NM}}_d\cap\mathrm{Rat}_d^{ss}$.\par 
We first state the following results, which give criteria to determine the (semi)stability of a point in $\mathbb{P}^{2d+1}$. 
\begin{lemma}\cite[Lemma 4.2]{DeMarco07}\label{iterate-stable-imply-stable}
Suppose $f\in\mathbb{P}^{2d+1}\setminus I(d)$ and $f^n$ is stable for some $n>1$. Then $f$ is stable.
\end{lemma}
Recall $\mu_f$ is the measure associated to $f\in\mathbb{P}^{2d+1}$.
\begin{lemma}\cite[Propositions 3.2 and 3.3]{DeMarco07}\label{semistable-measure}
Suppose $f\in\mathbb{P}^{2d+1}\setminus I(d)$. Then $f^n\in\mathrm{Rat}_{d^n}^{ss}$ for all $n\ge 1$ if and only if $\mu_{f}(\{z\})\le 1/2$ for all $z\in\mathbb{P}^1$.
\end{lemma}
Now for Newton maps, we have
\begin{proposition}\label{Newton-iterate-stable}
Fix $d\ge 2$. Let $N=H_N\widehat N\in\overline{\mathrm{NM}}_d$. Then the following are equivalent:
\begin{enumerate}
\item $N\in\mathrm{Rat}_d^s$,
\item for any $n\ge 1$, $\Psi_n(N)\in\mathrm{Rat}_{d^n}^s$,
\item for some $n\ge 1$, $\Psi_n(N)\in\mathrm{Rat}_{d^n}^s$.
\end{enumerate}
\end{proposition}
\begin{proof}
First we prove the case when $d$ is even. Similar argument works for odd $d$\par
$(1)\Rightarrow(2)$ If $N=H_N\widehat N\in\overline{\mathrm{NM}}_d\cap\mathrm{Rat}_d^s$, then $\Psi_n(N)$ is well-defined. Note each hole of $N$ is a fixed point of $\hat N$ with multiplicity $1$. If $z\in\mathbb{P}^1$ is a hole of $N$, by Lemma \ref{Newton-hole}, we have
$$d_z(N^n)\le d^{n-1}(\frac{d}{2}-1)+\sum_{k=1}^{n-1}(\frac{d}{2}-1)=(\frac{d}{2}-1)(d^{n-1}+n-1)\le\frac{d^n}{2}-1.$$
If $z\in\mathbb{P}^1$ is not a hole of $N$, we have 
\begin{align*}
d_z(N^n)&\le\sum_{k=1}^{n-1}(d-1)^k(\frac{d}{2}-1)\le\frac{d^n}{2}-1.
\end{align*}
Thus, by Lemma \ref{stable-hole}, the map $N^n$ is stable.\par 
$(3)\Rightarrow(1)$. It follows immediately from Lemma \ref{iterate-stable-imply-stable}.
\end{proof}
\begin{proposition}\label{Newton-iterate-semi}
Fix $d\ge 2$. Let $N=H_N\widehat N\in\overline{\mathrm{NM}}_d$. Then the following are equivalent:
\begin{enumerate}
\item $N\in\mathrm{Rat}_d^{ss}$,
\item for any $n\ge 1$, $\Psi_n(N)\in\mathrm{Rat}_{d^n}^{ss}$,
\item for some $n\ge 1$, $\Psi_n(N)\in\mathrm{Rat}_{d^n}^{ss}$.
\end{enumerate}
\end{proposition}
\begin{proof}
If $d$ is even, then $\mathrm{Rat}_d^s=\mathrm{Rat}_d^{ss}$. The equivalence of $(1)-(3)$ follows from Proposition \ref{Newton-iterate-stable}. Now we prove the equivalence when $d$ is odd.\par 
$(1)\Rightarrow(2)$ Consider the measure $\mu_N$. For $z\in\mathbb{P}^1$, we have 
$$\mu_N(\{z\})\le\frac{d-1}{2}\sum_{k=1}^\infty\frac{1}{d^k}=\frac{1}{2}.$$
Thus, by Lemma \ref{semistable-measure}, we know $N^n$ is semistable.\par 
$(3)\Rightarrow(1)$ Since $N^n$ is semistable, $d_z(N^n)\le(d^n+1)/2$ for all $z$. Therefore
$$\frac{d_z(N)}{d}\le\frac{d_z(N^n)}{d^n}\le\frac{d^n+1}{2d^n}.$$
So $d_z(N)\le(d-1)/2$. Thus $N$ is semistable.
\end{proof}
\begin{corollary}
Fix $d\ge 3$. The following are equivalent:
\begin{enumerate}
\item for any $N\in\partial\overline{\mathrm{NM}}_d\cap\mathrm{Rat}_d^{ss}$ and any $n\ge 2$, $\Psi_n(N)\not\in\mathrm{Rat}_{d^n}^s$;
\item $d=3$.
\end{enumerate}
\end{corollary}
\begin{proof}
$(1)\Rightarrow(2)$ We first claim $\partial\overline{\mathrm{NM}}_d\cap\mathrm{Rat}_d^{s}=\emptyset$. Otherwise, choose $N\in\partial\overline{\mathrm{NM}}_d\cap\mathrm{Rat}_d^{s}$. Then by Proposition \ref{Newton-iterate-stable}, we have $\Psi_n(N)\in\mathrm{Rat}_{d^n}^{s}$. It is a contradiction. Therefore, $\partial\overline{\mathrm{NM}}_d\cap\mathrm{Rat}_d^{s}=\emptyset$. Hence $d$ is odd. Moreover, for any $N\in\partial\overline{\mathrm{NM}}_d\cap\mathrm{Rat}_d^{ss}$, $N$ has a hole of depth $(d-1)/2$. So $d=3$.\par
$(2)\Rightarrow(1)$ It follows from Corollary \ref{cubic-stable-empty} and Proposition \ref{Newton-iterate-semi}.
\end{proof}

\subsection{Iterate Maps on $\overline{\mathrm{nm}}_d$}
In this subsection, we study the iterate maps on $\overline{\mathrm{nm}}_d$ and prove Theorem \ref{theorem-iterate-continuous-moduli}.\par
Consider the iterate maps
$$\Phi_n:\mathrm{rat}_d\to\mathrm{rat}_{d^n},$$
sending $[f]$ to $[f^n]$. Then the iterate maps $\Phi_n$ induce rational maps 
$$\Phi_n:\overline{\mathrm{rat}}_d\dasharrow\overline{\mathrm{rat}}_{d^n}.$$
However, the iterate maps $\Phi_n$ are not regular on $\overline{\mathrm{rat}}_d$ for any $d\ge 2$ and $n\ge 2$ \cite[Section 10]{DeMarco07}.\par 
If we focus on the moduli spaces $\overline{\mathrm{nm}}_d$, we have
\begin{theorem}\label{GIT-continuous-entension}
The iterate maps
$$\Phi_n|_{\mathrm{nm}_d}:\mathrm{nm}_d\to\mathrm{rat}_{d^n},$$
sending $[f]$ to $[f^n]$, extend naturally to continuous maps from $\overline{\mathrm{nm}}_d$ to $\overline{\mathrm{rat}}_{d^n}$ for any $d\ge 2$ and $n\ge 2$.
\end{theorem}
\begin{proof}
We prove the case for odd $d$. Similar argument also works for $d$ even. 
Note by Proposition \ref{semi-singleton}, as sets, we have 
$$\overline{\mathrm{nm}}_d=\overline{\mathrm{NM}}_d\cap\mathrm{Rat}_d^s/\mathrm{Aut}(\mathbb{C})\cup\Big\{\big[X^{\frac{d-1}{2}}Y^{\frac{d-1}{2}}[(d-1)X:(d+1)Y]\big]_{\mathrm{GIT}}\Big\}.$$
To ease notations, we write 
$$\phi_d([X:Y])=X^\frac{d-1}{2}Y^\frac{d-1}{2}[(d-1)X:(d+1)Y].$$ 
Then 
$$\phi_d^n([X:Y])=X^\frac{d^n-1}{2}Y^\frac{d^n-1}{2}[(d-1)^nX:(d+1)^nY]\in\mathrm{Rat}_d^{ss}.$$ 
Define 
$$\Phi_{\overline{\mathrm{nm}}_d}([N]_{\mathrm{GIT}})=
\begin{cases}
[N^n]_{\mathrm{GIT}}\ \text{if}\ [N]_{\mathrm{GIT}}\in\overline{\mathrm{NM}}_d\cap\mathrm{Rat}_d^s/\mathrm{PSL}_2(\mathbb{C}),\\
[\phi_d^n]_{\mathrm{GIT}}\ \text{if}\ [N]_{\mathrm{GIT}}=[\phi_d]_\mathrm{GIT}.
\end{cases}$$\par
First $\Phi_n|_{\overline{\mathrm{nm}}_d}$ is well defined. Indeed, let $[f]_{\mathrm{GIT}}=[g]_{\mathrm{GIT}}\in\overline{\mathrm{nm}}_d$. Write $f=H_f\hat f$ and $g=H_g\hat g$. If $f, g\in\mathrm{Rat}_d^{s}$, then $g$ is conjugate to $f$. Without loss of generality, we assume $f\in\overline{\mathrm{NM}}_d\cap\mathrm{Rat}_d^s$. Then $\deg\hat f\ge 1$. Otherwise, $f([X:Y])=Y^d(1:0)\not\in\mathrm{Rat}_d^s$. So $\deg\hat g=\deg\hat f\ge 1$. Thus $f^n$ is conjugate to $g^n$. By Proposition \ref{Newton-iterate-stable}, we have $f^n\in\mathrm{Rat}_{d^n}^s$ and so is $g^n$. Hence $[f^n]_{\mathrm{GIT}}=[g^n]_{\mathrm{GIT}}\in\overline{\mathrm{rat}}_{d^n}$. If $f,g\in\mathrm{Rat}_d^{ss}\setminus\mathrm{Rat}_d^s$, by Proposition \ref{semi-singleton}, 
$$[f]_{\mathrm{GIT}}=[g]_{\mathrm{GIT}}=[\phi_d]_{\mathrm{GIT}}.$$ 
Then we have
$$\Phi_n|_{\overline{\mathrm{nm}}_d}([f]_{\mathrm{GIT}})=\Phi_n|_{\overline{\mathrm{nm}}_d}([g]_{\mathrm{GIT}})\in\overline{\mathrm{rat}}_{d^n}.$$
Thus, $\Phi_n|_{\overline{\mathrm{nm}}_d}$ defines a map from $\overline{\mathrm{\mathrm{nm}}}_d$ to $\overline{\mathrm{rat}}_{d^n}$.\par 
Now we show the map $\Phi_n|_{\overline{\mathrm{nm}}_d}$ is continuous. Let $\{[f_k]_{\mathrm{GIT}}\}$ be a sequence in $\overline{\mathrm{nm}}_d$ such that $[f_k]_{\mathrm{GIT}}\to[f]_{\mathrm{GIT}}\in\overline{\mathrm{nm}}_d$ as $k\to\infty$. Then there are representatives $N_k\in\overline{\mathrm{NM}}_d\cap\mathrm{Rat}_d^{ss}$ and $N\in\overline{\mathrm{NM}}_d\cap\mathrm{Rat}_d^{ss}$ for $[f_k]_{\mathrm{GIT}}$ and $[f]_{\mathrm{GIT}}$, respectively. By Proposition \ref{Newton-Seim-Indeter}, $N_k^n$ and $N^n$ are well defined. Moreover, by Proposition \ref{Newton-Seim-Indeter}, we have $N_k^n$ converges to $N^n$, as $k\to\infty$,. By Proposition \ref{Newton-iterate-semi}, we know $[N_k^n]_{\mathrm{GIT}}$ and $[N^n]_{\mathrm{GIT}}$ are in $\overline{\mathrm{rat}}_{d^n}$. Thus we have $[N_k^n]_{\mathrm{GIT}}$ converges to $[N^n]_{\mathrm{GIT}}$, as $k\to\infty$.\par 
If $N\in\overline{\mathrm{NM}}_d\cap\mathrm{Rat}_d^s$, then $\Phi_n|_{\overline{\mathrm{nm}}_d}$ is continuous at $[N]_{\mathrm{GIT}}$. Hence $\Phi_n|_{\overline{\mathrm{nm}}_d}$ is continuous at $[f]_{\mathrm{GIT}}$.\par 
If $N\in\overline{\mathrm{NM}}_d\cap(\mathrm{Rat}_d^{ss}\setminus\mathrm{Rat}_d^s)$, then 
$$[N]_{\mathrm{GIT}}=[\phi_d]_{\mathrm{GIT}}\in\overline{\mathrm{nm}}_d.$$ Thus there exists $M_t\in\mathrm{PSL}_2(\mathbb{C})$ such that as $t\to 0$, 
$$M_t^{-1}\circ N\circ M_t\to\phi_d.$$
Note $\phi_d\not\in I(d)$. Then we have 
$$M_t^{-1}\circ N^n\circ M_t\to\phi_d^n.$$
Note $\phi_d^n\in\mathrm{Rat}_{d^n}^{ss}$. Thus $[N^n]_{\mathrm{GIT}}=[\phi_d^n]_{\mathrm{GIT}}$ in $\overline{\mathrm{rat}}_d$. Thus, in this case, $\Phi_n|_{\overline{\mathrm{nm}}_d}$ is continuous at $[N]_{\mathrm{GIT}}$. Hence $\Phi_n|_{\overline{\mathrm{nm}}_d}$ is continuous at $[f]_{\mathrm{GIT}}$.
\end{proof}

\section{Applications}\label{applications}
In this section, we consider the inverse limits compactifications $\overleftarrow{\mathrm{nm}}_d$ and the maximal measures compactifications $\widetilde{\mathrm{nm}}_d$ for the moduli spaces $\mathrm{nm}_d$, and prove Theorem \ref{theorem-inverse-limit-compactify}.
\subsection{Inverse Limits Compactifications}
Recall that Theorem \ref{GIT-continuous-entension} states for any $d\ge 2$ and $n\ge 2$, the iterate maps 
$$\Phi_n|_{\overline{\mathrm{nm}}_d}:\overline{\mathrm{nm}}_d\to\overline{\mathrm{rat}}_{d^n},$$ 
sending $[f]_\mathrm{GIT}$ to $[f^n]_\mathrm{GIT}$, are continuous. Then 
$$\overline{\Phi_n(\mathrm{nm}_d)}=\Phi_n(\overline{\mathrm{nm}}_d).$$
Now consider the maps 
$$\vec{\Phi}_n:=(\mathrm{Id},\Phi_2,\cdots,\Phi_n):\mathrm{rat}_d\to\mathrm{rat}_d\times\mathrm{rat}_{d^2}\times\cdots\times\mathrm{rat}_{d^n}.$$
Note $\vec{\Phi}_n|_{\mathrm{nm}_d}$ extends to a continuous map on $\overline{\mathrm{nm}}_d$. Let $\Gamma_n$ be the closure of $\vec{\Psi}_n(\mathrm{nm}_d)$ in $\overline{\mathrm{rat}}_d\times\overline{\mathrm{rat}}_{d^2}\times\cdots\times\overline{\mathrm{rat}}_{d^n}$. Then 
$$\Gamma_n=\overline{\vec{\Phi}_n(\mathrm{nm}_d)}=\vec{\Phi}_n(\overline{\mathrm{nm}}_d).$$
There is a natural projection from $\Gamma_{n+1}$ to $\Gamma_n$. We take the inverse limit over $n$, 
$$\overleftarrow{\mathrm{nm}}_d:=\lim\limits_{\longleftarrow}\Gamma_n.$$
Then $\overleftarrow{\mathrm{nm}}_d$ is compact. For more details about the inverse limits, we refer \cite{Stone79}. By identifying $[N]\in\mathrm{nm}_d$ with $(([N]),([N],[N^2]),\cdots)\in\overleftarrow{\mathrm{nm}}_d$, the moduli space $\mathrm{nm}_d$ is a dense open subset of $\overleftarrow{\mathrm{nm}}_d$. Furthermore,
\begin{theorem}
The map 
$$\widehat\Phi:\overline{\mathrm{nm}}_d\to\overleftarrow{\mathrm{nm}}_d,$$
given by $[N]_\mathrm{GIT}\to(\vec{\Phi}_1([N]_\mathrm{GIT}),\vec{\Phi}_2([N]_\mathrm{GIT}),\cdots)$, is a well-defined homeomorphism.
\end{theorem}
\begin{proof}
If $[f]_{\mathrm{GIT}}=[g]_{\mathrm{GIT}}\in\overline{\mathrm{nm}}_d$, then by Theorem \ref{GIT-continuous-entension}, we have $[f^n]_{\mathrm{GIT}}=[g^n]_{\mathrm{GIT}}\in\overline{\mathrm{rat}}_{d^n}$ for any $n\ge 2$. Thus $\vec{\Phi}_n([f]_\mathrm{GIT})$ is well defined. Therefore, $\widehat\Phi([f]_{\mathrm{GIT}})=\widehat\Phi([g]_{\mathrm{GIT}})$. So $\widehat\Phi$ is well-defined.\par
Note $\widehat\Phi$ is bijective and continuous. To show the map $\widehat\Phi$ is a homeomorphism, we only need to show $\overleftarrow{\mathrm{nm}}_d$ is Hausdorff. \par
Note 
$$\widehat\Phi(\overline{\mathrm{NM}}_d\cap\mathrm{Rat}_d^s/\mathrm{PSL}_2(\mathbb{C}))\subset\prod_{i=1}^\infty\prod_{n=1}^i\mathrm{rat}^{s}_{d^n}$$
Then $\widehat\Phi(\overline{\mathrm{NM}}_d\cap\mathrm{Rat}_d^s/\mathrm{PSL}_2(\mathbb{C}))$ is Hausdorff since $\mathrm{rat}^{s}_{d^n}$ is Hausdorff for any $n\ge 1$. So for even $d\ge 2$, the space $\overleftarrow{\mathrm{nm}}_d$ is Hausdorff.\par
If $d\ge 3$ is odd, by Proposition \ref{semi-singleton}, we have 
$$\overleftarrow{\mathrm{nm}}_d=\widehat\Phi(\overline{\mathrm{nm}}_d)=\widehat\Phi(\overline{\mathrm{NM}}_d\cap\mathrm{Rat}_d^s/\mathrm{PSL}_2(\mathbb{C}))\cup\{\widehat\Phi([X^{\frac{d-1}{2}}Y^{\frac{d-1}{2}}[(d-1)X:(d+1)Y]]_{\mathrm{GIT}})\}.$$ 
So $\overleftarrow{\mathrm{nm}}_d$ is the one-point compactification of $\widehat\Phi(\overline{\mathrm{NM}}_d\cap\mathrm{Rat}_d^s/\mathrm{PSL}_2(\mathbb{C}))$. Since the geometric quotient $\mathrm{rat}_d^s$ is Hausdorff and locally compact, $\widehat\Phi(\overline{\mathrm{NM}}_d\cap\mathrm{Rat}_d^s/\mathrm{PSL}_2(\mathbb{C}))$ is also Hausdorff and locally compact. Thus $\overleftarrow{\mathrm{nm}}_d$ is Hausdorff.
\end{proof}
\begin{corollary}
Every element of $\overleftarrow{\mathrm{nm}}_d$ is uniquely determined by the first entry.
\end{corollary}
\begin{corollary}
For the cubic Newton maps, $\overleftarrow{\mathrm{nm}}_3$ is homeomorphic to $\mathbb{P}^1$.
\end{corollary}

\subsection{Maximal Measure Compactifications}
Recall the associated measure for $f\in\mathbb{P}^{2d+1}$ is
$$\mu_f=
\begin{cases}
\lim\limits_{n\to\infty}\frac{1}{d^n}\sum\limits_{\{f^n(z)=a\}}\delta_z, &\text{if}\  f\in\mathrm{Rat}_d\ \text{and}\  a\in\mathbb{P}^1\ \text{is\  nonexceptional},\\
\sum\limits_{n=0}^\infty\frac{1}{d^{n+1}}\sum\limits_{\{H_f(a)=0\}}\sum\limits_{\{\hat f^n(z)=a\}}\delta_z, &\text{if}\ f=H_f\hat f\in\mathbb{P}^{2d+1}\setminus\mathrm{Rat}_d\ \text{and}\ \deg\hat f>0,\\
\frac{1}{d}\sum\limits_{\{H_f(a)=0\}}\delta_a, &\text{if}\ f=H_f\hat f\in\mathbb{P}^{2d+1}\setminus\mathrm{Rat}_d\ \text{and}\ \deg\hat f=0.
\end{cases}$$\par
For $f=H_f\hat f\not\in I(d)$, we define the Fatou set $F(f)$ of $f$ as the largest open set on which iterates of $f$ form a normal family. The complement of $F(f)$ is called the Julia set of $f$ and denote by $J(f)$. Note if a subset $U\subset\mathbb{P}^1$ contains a hole of $f$, then $U$ intersects the Julia set $J(f)$. Thus for $f=H_f\hat f$ with $\text{deg}\ \hat f>0$, we have
$$J(f)=J(\hat f)\cup\overline{\bigcup_{n=0}^\infty\bigcup_{\{H_f(a)=0\}}\hat f^{-n}(a)}.$$
If $\text{deg}\ \hat f=0$, set $J(f)=\emptyset$.\par 
\begin{lemma} \cite[Proposition 3.2]{DeMarco07}
Let $f=H_f\hat f\notin I(d)$ be degenerate. If at least one hole of $f$ is nonexceptional for $\hat f$, then $\mathrm{supp}\ \mu_f=J(f)$. If each hole of $f$ is exceptional for $\hat f$, then $\mathrm{supp}\ \mu_f$ is contained in the exceptional set of $\hat f$.
\end{lemma}
Let $M^1(\mathbb{P}^1)$ be the space of probability measures on $\mathbb{P}^1$ with the weak$-\ast$ topology. Identify $\mathbb{S}^2\subset\mathbb{R}^3$ as $\mathbb{P}^1$ via stereographic projection, and use the unit ball in $\mathbb{R}^3$ as a model for the hyperbolic three space $\mathbb{H}^3$. The Euclidean center of mass of a probability measure $\mu\in M^1(\mathbb{P}^1)$ on $\mathbb{S}^2$ is given by
$$E(\mu)=\int_{\mathbb{S}^2}z d\mu(z).$$
Given $\mu\in M^1(\mathbb{P}^1)$ such that $\mu(\{z\})<1/2$ for all $z\in\mathbb{P}^1$, the conformal barycenter $C(\mu)\in\mathbb{H}^3$ is uniquely determined by the following two properties \cite{Douady86}:\par
(1) $C(\mu)=0$ in $\mathbb{R}^3$ if and only if $E(\mu)=0$, and\par
(2) $C(A_\ast\mu)=A(C(\mu))$ for all $A\in\mathrm{PSL}_2(C)$.\\
The barycenter is a continuous function on the space of probability measures such that $\mu(\{z\})<1/2$ for all $z\in\mathbb{P}^1$, and it is undefined if $\mu$ has an atom of mass at least $1/2$. A measure $\mu$ is said to be the barycentered if $C(\mu)=0$.
\begin{lemma}\label{Newton-barycenter}
For $N=H_N\widehat N\in\overline{\mathrm{NM}}_d\cap\mathrm{Rat}^{s}_d$, the barycenter $C(\mu_N)$ is well-defined.
\end{lemma}
\begin{proof}
If $N\in\mathrm{NM}_d$, then $\mu_N$ is atomless and hence $C(\mu_N)$ is well-defined. \par
If $N=H_N\widehat N\in\partial\overline{\mathrm{NM}}_d\cap\mathrm{Rat}^{s}_d$, then by Lemma \ref{Newton-hole} the depth of each hole of $N$ is $\le d/2-1$ if $d$ is even, and $\le(d-3)/2$ if $d$ is odd. Note $\deg\hat N\ge 1$. Thus for any $w\in\mathbb{P}^1$, we have 
$$\mu_N(\{w\})=\sum\limits_{n=0}^\infty\frac{1}{d^{n+1}}\sum\limits_{\{H_f(a)=0\}}\sum\limits_{\{\hat f^n(z)=a\}}\delta_z(\{w\})\le\sum\limits_{n=0}^\infty\frac{1}{d^{n+1}}\frac{d-2}{2}<\frac{1}{2}.$$
Thus $C(\mu_N)$ is well-defined.
\end{proof}
\begin{remark}
If $N\in\overline{\mathrm{NM}}_d\cap(\mathrm{Rat}^{ss}_d\setminus\mathrm{Rat}_d^s)$, then $C(\mu_N)$ is undefined. Indeed, in this case, $f$ has a hole $h$ of depth $(d-1)/2$ and hence $\mu_N(\{h\})=1/2$.
\end{remark}
For $[f]\in\mathrm{rat}_d$, we can choose a normalized representative $f_{bc}\in\mathrm{Rat}_d$ such that $[f]=[f_{bc}]\in\mathrm{rat}_d$ and $C(\mu_{f_{bc}})=0$. Then $f_{bc}$ is called a conformal barycentered representative of $f$. Note $\mu_{f_{bc}}$ is unique up to conjugation by $\mathrm{SO}(3)\subset\mathrm{PSL}_2(\mathbb{C})$. For $f=H_f\hat f\in\overline{\mathrm{NM}}_d\cap\mathrm{Rat}_d^{s}$, by Lemma \ref{Newton-barycenter}, we can pick a conformal barycentered representative $f_{bc}\in\mathrm{Rat}^s_d$, i.e. $[f]=[f_{bc}]\in\mathrm{rat}^s_d$ and $C(\mu_{f_{bc}})=0$. \par
Let $M^1_{bc}$ be the space of barycentered probability measures on $\mathbb{P}^1$. Note the space $M^1_{bc}$ is not closed. Let $\overline{M^1_{bc}}$ be the closure of $M^1_{bc}$ in the space of all probability measures.  Consider the map 
$$\Theta:\mathrm{rat}_d\to M^1_{bc}/\mathrm{SO}(3),$$
given by $\Theta([f])=[\mu_{f_{bc}}]$, then $\Theta$ is continuous. Let $\widetilde{\mathrm{rat}}_d$ be the closure of the graph of $\Theta$, that is 
$$\widetilde{\mathrm{rat}}_d:=\overline{\mathrm{Graph}(\Theta)}\subset\overline{\mathrm{rat}}_d\times\overline{M^1_{bc}}/\mathrm{SO}(3).$$
Then $\widetilde{\mathrm{rat}}_d$ is compact and contains $\mathrm{rat}_d$ as a dense open subset.\par
Consider the restriction $\Theta|_{\mathrm{nm}_d}$ and the closure of the graph of $\Theta|_{\mathrm{nm}_d}$, 
$$\widetilde{\mathrm{nm}}_d:=\overline{\mathrm{Graph}(\Theta|_{\mathrm{\mathrm{nm}}_d})}\subset\overline{\mathrm{nm}}_d\times\overline{M^1_{bc}}/\mathrm{SO}(3).$$
Then $\widetilde{\mathrm{nm}}_d$ is compact and contains $\mathrm{nm}_d$ as a dense open subset. In fact, $\widetilde{\mathrm{nm}}_d$ is the closure of $\mathrm{nm}_d$ in $\widetilde{\mathrm{rat}}_d$.\par
To find out $\overline{\Theta(\mathrm{nm}_d)}\subset\overline{M^1_{bc}}/\mathrm{SO}(3)$, we first state the following known results.
\begin{lemma}\cite[Lemma 8.3]{DeMarco07}
If $\mu\in\overline{M^1_{bc}}$, then either $\mu$ is barycentered or 
$$\mu=\frac{1}{2}\delta_a+\frac{1}{2}\delta_{-1/\bar{a}}.$$
\end{lemma}
\begin{lemma}\cite[Theorem 8.1]{DeMarco07}
\begin{enumerate}
\item $M^1_{bc}/\mathrm{SO}(3)$ is locally compact and Hausdorff.
\item $\overline{M^1_{bc}}/\mathrm{SO}(3)$ is the one-point compactification of $M^1_{bc}/\mathrm{SO}(3)$.
\end{enumerate}  
\end{lemma}
\begin{lemma}\cite[Lemma 8.4]{DeMarco07}\label{Mea-Mod}
Suppose $\{\mu_k\}\subset M^1_{bc}$ such that $\mu_k\to\mu$ weakly, and $A_k$ is a sequence of round annuli such that
\begin{enumerate}
\item $\mathrm{Mod}(A_k)\to\infty$ as $k\to\infty$, and
\item there is a sequence $\epsilon_k\to 0$ such that $\mu_k(D_k)\ge 1/2-\epsilon$ for each of the complementary disks $D_k$ of $A_k$.
\end{enumerate}
Then $[\mu]=\infty$ in $\overline{M^1_{bc}}/\mathrm{SO}(3)$.
\end{lemma}
Then we can get 
\begin{proposition}\label{Max-Mea-Continuous}
Let $\theta=\Theta|_{\mathrm{nm}_d}$. Then $\theta$ extends to a continuous map 
$$\overline{\theta}:\overline{\mathrm{nm}}_d\to\overline{M^1_{bc}}/\mathrm{SO}(3).$$
\end{proposition}
\begin{proof}
Define $\overline{\theta}:\overline{\mathrm{nm}}_d\to\overline{M^1_{bc}}/\mathrm{SO}(3)$ by
$$\overline{\theta}([N])=
\begin{cases}
\theta([N]_{\mathrm{\mathrm{GIT}}}), &\text{if}\ [N]_{\mathrm{GIT}}\in\mathrm{nm}_d,\\
[\mu_{N_{bc}}], &\text{if}\ [N]_{\mathrm{GIT}}\in\partial\overline{\mathrm{NM}}_d\cap\mathrm{Rat}_d^{s}/\mathrm{PSL}_2(\mathbb{C}),\\
[\delta_0/2+\delta_\infty/2]&\text{if}\ [N]_{\mathrm{GIT}}\in\partial\overline{\mathrm{NM}}_d\cap(\mathrm{Rat}_d^{ss}\setminus\mathrm{\mathrm{Rat}}_d^{s})//\mathrm{PSL}_2(\mathbb{C}).
\end{cases}$$\par
We first show $\overline{\theta}$ is well-defined. It only needs to show if $[f]=[g]\in\partial\overline{\mathrm{NM}}_d\cap\mathrm{Rat}_d^{s}/\mathrm{PSL}_2(\mathbb{C})$, then $[\mu_{f_{bc}}]=[\mu_{g_{bc}}]\in\overline{M^1_{bc}}/\mathrm{SO}(3)$. Note $[f_{bc}]=[g_{bc}]\in\overline{\mathrm{rat}}_d$. Thus, $[\mu_{f_{bc}}]=[\mu_{g_{bc}}]\in\overline{M^1_{bc}}/\mathrm{SO}(3)$.\par
Now we show $\overline{\theta}$ is continuous. For $[f_n]_{\mathrm{GIT}}\in\overline{\mathrm{nm}}_d$, suppose $[f_n]_{\mathrm{GIT}}$ converges to $[f]_{\mathrm{GIT}}$ in $\overline{\mathrm{nm}}_d$. We have three cases.\par
(1) If $[f]_{\mathrm{GIT}}\in\mathrm{nm}_d$, it is done by the continuity of $\theta$.\par
(2) If $[f]_{\mathrm{GIT}}\in\partial\overline{\mathrm{NM}}_d\cap\mathrm{Rat}_d^{s}/\mathrm{PSL}_2(\mathbb{C})$, then 
$$[(f_n)_{bc}]\to[f_{bc}]\in\partial\overline{\mathrm{NM}}_d\cap\mathrm{Rat}_d^s/\mathrm{PSL}_2(\mathbb{C}).$$ 
So the accumulation points of $\{(f_n)_{bc}\}$ are conjugate to $f_{bc}$. Hence for a convergent subsequence $\{(f_{n_i})_{bc}\}$, there exists $M\in\mathrm{PSL}_2(\mathbb{C})$ such that $(f_{n_i})_{bc}$ converges to $M^{-1}\circ f_{bc}\circ M\not\in I(d)$. Then by Proposition \ref{measure-converge}, we know $\mu_{(f_{n_i})_{bc}}\to\mu_{M^{-1}\circ f_{bc}\circ M}$ weakly. So $[\mu_{(f_{n_i})_{bc}}]\to[\mu_{f_{bc}}]$. Hence $[\mu_{(f_{n})_{bc}}]\to[\mu_{f_{bc}}]$. Thus $[\mu_{(f_{n_i})_{bc}}]\to[\mu_{f_{bc}}]$ in this case.\par
(3) If $[f]_{\mathrm{GIT}}\in\partial\overline{\mathrm{NM}}_d\cap(\mathrm{Rat}_d^{ss}\setminus\mathrm{Rat}_d^{s})//\mathrm{PSL}_2(\mathbb{C})$, then $d$ is odd. By Lemma \ref{Newton-hole}, there exists a hole $h\in\mathbb{P}^1$ of $f=H_f\hat f$ such that $d_h(f)=(d-1)/2$ and $\hat f(h)=h$. Thus every accumulation point of $f_n$ has a hole of depth $(d-1)/2$. Suppose $g=H_{g}\hat g$ is an accumulation point of $f_n$ with $d_{a}(g)=(d-1)/2$ and $\hat g(a)=a$. Then 
$$\mu_{g}(\{a\})=\sum\limits_{n=0}^\infty\frac{1}{d^{n+1}}\sum\limits_{\{H_{g}(w)=0\}}\sum\limits_{\{\hat{g}^n(z)=w\}}\delta_z(\{a\})=\sum\limits_{n=0}^\infty\frac{1}{d^{n+1}}\frac{d-1}{2}=\frac{1}{2}.$$
Now for fixed $\epsilon>0$, choose $r=r(\epsilon)<1$ such that 
$$\mu_{g}(\mathbb{P}^1\setminus\overline{B}(a,r))\ge\frac{1}{2}-\frac{\epsilon}{2}.$$ 
Since $\mu_{f_{n_k}}\to\mu_{g}$ weakly as $k\to\infty$, there exists $K(\epsilon)>0$ such that for all $k\ge K(\epsilon)$,
$$\mu_{f_{n_k}}(\mathbb{P}^1\setminus\overline{B}(a,r))\ge\frac{1}{2}-\epsilon$$
and 
$$\mu_{f_{n_k}}(B(a,r^2))\ge\frac{1}{2}-\epsilon.$$  
For given $k$, let $\epsilon_k$ be the smallest $\epsilon$ such that $k\ge K(\epsilon)$ and let $r_k=r(\epsilon_k)$. Set $A'_{n_k}=B(a,r_k)\setminus\overline{B}(a,r_k^2)$, then $\mathrm{Mod}(A'_{n_k})\to\infty$ as $k\to\infty$. Let $A_{n_k}=M^{-1}(A_{n_k})$. Then $\mathrm{Mod}(A_{n_k})\to\infty$ as $k\to\infty$. Let $\mu$ be any subsequential limit of $\mu_{f_{n_k}}$. By Lemma \ref{Mea-Mod}, we have $[\mu]=[\delta_0/2+\delta_\infty/2]$. So in this case $\overline{\theta}([f_n]_{\mathrm{GIT}})\to[\delta_0/2+\delta_\infty/2]$. Therefore, $\overline{\theta}$ is continuous.
\end{proof}
\begin{corollary}\label{Closure-Max-Mea}
For $d\ge 2$,
$$\overline{\Theta}(\mathrm{nm}_d)={\Theta(\mathrm{nm}_d)}\cup\bigcup\limits_{f\in\partial\overline{\mathrm{NM}}_d\cap\mathrm{Rat}_d^s}\{[\mu_{f_{bc}}]\}\cup\{[\delta_0/2+\delta_\infty/2]\}.$$
\end{corollary}
\begin{corollary}\label{Max-Mea}
For $d\ge 2$,
\begin{align*}
\widetilde{\mathrm{nm}}_d=&\mathrm{Graph}(\Theta|_{\mathrm{nm}_d})\cup\bigcup\limits_{f\in\partial\overline{\mathrm{NM}}_d\cap\mathrm{Rat}_d^s}\{([f],[\mu_{f_{bc}}])\}\\
&\cup\{([X^{\frac{d-1}{2}}Y^{\frac{d-1}{2}}[(d-1)X:(d+1)Y]]_{\mathrm{GIT}},[\delta_0/2+\delta_\infty/2])\}.
\end{align*}
\end{corollary}
Proposition \ref{Max-Mea-Continuous} implies $\overline{\mathrm{Graph}(\Theta|_{\mathrm{nm}_d})}=\mathrm{Graph}(\overline{\Theta|_{\mathrm{nm}_d}})$. Thus we have
\begin{theorem}\label{Max-Mea-GIT}
The maximal measure compactifiction $\widetilde{\mathrm{nm}_d}$ is homeomorphic to the GIT compactification $\overline{\mathrm{nm}}_d$.
\end{theorem}
\begin{corollary}
The maximal measure compactifiction $\widetilde{\mathrm{nm}}_d$ is homeomorphic to the inverse limit compactification $\overleftarrow{\mathrm{nm}}_d$.
\end{corollary}
\begin{corollary}
For $d=3$,
\begin{enumerate}
\item $\overline{\Theta}(\overline{\mathrm{nm}}_3)={\Theta}(\mathrm{nm}_3)\cup\{[\delta_0/2+\delta_\infty/2]\}$.
\item $\widetilde{\mathrm{nm}}_3=\mathrm{Graph}(\Theta|_{\mathrm{nm}_3})\cup\{([XY[X:2Y]]_{\mathrm{GIT}},[\delta_0/2+\delta_\infty/2])\}$.
\item $\widetilde{\mathrm{nm}}_3$ is homeomorphic to $\mathbb{P}^1$.
\end{enumerate}
\end{corollary}
\begin{proof}By Corollary \ref{cubic-stable-empty}, we know $\partial\overline{\mathrm{NM}}_3\cap\mathrm{Rat}^s_d=\emptyset$. Note $\overline{\mathrm{nm}}_3$ is homeomorphic to $\mathbb{P}^1$. Then the results follow from Corollary \ref{Closure-Max-Mea}, Corollary \ref{Max-Mea} and Proposition \ref{Max-Mea-GIT}.
\end{proof}

\section{Deligne-Mumford Compactifications}\label{Deligne-Mumford}
\subsection{Moduli Spaces of Marked Riemann Surfaces}
In this subsection, we give the background of moduli spaces of marked Riemann surfaces and the corresponding Deligne-Mumford compactifications.\par
Let $\Sigma_{g,n}$ be a Riemann surface of genus $g$ with $n$ marked points. For $2-2g-n<0$, the moduli space $\mathcal{M}_{g,n}$ is the set of isomorphism classes $[\Sigma_{g,n}]$ of Riemann surfaces of genus $g$ with $n$ marked points. The automorphism group of any Riemann surface satisfying $2-2g-n<0$ is finite \cite[Section V.1]{Farkas92}.\par
Here, we are interested only in the case $g=0$. We give some examples to illustrate the moduli spaces $\mathcal{M}_{0,n}$. For more examples, we refer \cite[Section 1]{Zvonkine12}.
\begin{example}
Let $n=3$. Any Riemann surface of genus $0$ with three marked points can be identified with $(\mathbb{P}^1, 0, 1,\infty)$. Thus $\mathcal{M}_{0,3}$ is a singleton.
\end{example}
\begin{example}
Let $n=4$. Any Riemann surface of genus $0$ with four marked points $(x_1,x_2,x_3,x_4)$ can be uniquely identified with $(\mathbb{P}^1, 0, 1,\infty, t)$, where $t\in\mathbb{P}^1\setminus\{0,1,\infty\}$ is determined by the positions of marked points on the original Riemann surface. If the original Riemann surface is $\mathbb{P}^1$, then $t$ is the cross-ratio of $x_1,x_2,x_3,x_4$. Thus the moduli space $\mathcal{M}_{0,4}$ is the values of $t$. Hence $\mathcal{M}_{0,4}$ can be identified with $\mathbb{P}^1\setminus\{0,1,\infty\}$.
\end{example}
To compactify the moduli space $\mathcal{M}_{0,n}$, we need to add some new points that are called ``stable curves". We define the Riemann surfaces with nodes, which were first introduced by Bers \cite{Bers74}.
\begin{definition}
A surface $\Sigma$ of genus $0$ with nodes is a Hausdorff space whose every point has a neighborhood homeomorphic either to a disk in the complex plane or to 
$$U=\{(z,w)\in\mathbb{C}^2:zw=0, |z|<1, |w|<1\}.$$
\end{definition}
Following Bers \cite{Bers74}, a point $p$ in the surface $\Sigma$ with nodes is a node if every neighborhood of $p$ contains an open set homeomorphic to $U$. Each component of the complement of the nodes of $\Sigma$ is called a part of $\Sigma$. Then each part is a Riemann surface of genus $0$  and compact except for punctures.\par
The stable curves genus of $g=0$ are defined as following, see \cite[Definition 1.25]{Zvonkine12}.
\begin{definition} 
A \textit{stable curve} $\mathcal{C}$ of genus $g=0$ with $n$ marked points is a complex algebraic curve such that
\begin{enumerate}
\item the only singularities of $\mathcal{C}$ are simple nodes ,
\item the marked points are distinct and none are nodes,
\item the surface obtained from $\mathcal{C}$ by resolving all its nodes is of genus $0$,
\item for each part, the number of marked points and nodes is at least $3$.
\end{enumerate}
\end{definition}
From the definition, we know a stable curve $\mathcal{C}$ of genus $0$ with $n$ marked points has a finite number of automorphisms.\par
The Deligne-Mumford compactification $\widehat{\mathcal{M}}_{0,n}$ of the moduli space $\mathcal{M}_{0,n}$ was introduced by P. Deligne and D. Mumford \cite{Deligne69}. It is smooth compact complex $n$-dimensional irreducible projective variety \cite{Harris98}. In the algebraic and analytic categories, it is a coarse moduli space for the stable curves functor \cite{Hubbard14}. The set 
$$\partial\mathcal{M}_{0,n}:=\widehat{\mathcal{M}}_{0,n}\setminus\mathcal{M}_{0,n}$$
parametrizing singular stable curves is called the boundary of $\mathcal{M}_{0,n}$. We refer \cite{Funahashi12} and \cite[Section 4.3]{Salamon99} for the topology on the spaces $\widehat{\mathcal{M}}_{0,n}$ .
\begin{example}
The space $\mathcal{M}_{0,3}$ is a singleton. Hence $\widehat{\mathcal{M}}_{0,3}$ is also a singleton.
\end{example} 
\begin{example}
The boundary $\partial\mathcal{M}_{0,4}$ consists of three singular stable rational curves with $4$ marked points. Each of these curves are two copies of $\mathbb{P}^1$ glued at a point, with marked points.
\end{example} 
We can define the affine moduli space $\mathcal{M}^\ast_{0,n}$ consisting of the affine conjugate classes $[(x_1,\cdots,x_n)]$ of ordered $n$ distinct points in $\mathbb{C}$, that is 
$$\mathcal{M}^\ast_{0,n}=\{(x_1,\cdots,x_n)\in\mathbb{C}^n\setminus\Delta_n\}/\mathrm{Aut}(\mathbb{C}),$$
where $\Delta_n=\{(z_1,\cdots,z_n): \exists i\not=j\ \text{such that}\ z_i=z_j\}\subset\mathbb{C}^n$ is the diagonal set.
Then, for example, $\mathcal{M}^\ast_{0,n}$ is a singleton for $n=1$ and $2$, and 
$$\mathcal{M}^\ast_{0,3}=\mathbb{C}\setminus\{0,1\}.$$
In general, we have
\begin{lemma}
For $n\ge 1$, the map 
$$\mathcal{I}:\mathcal{M}^\ast_{0,n}\to\mathcal{M}_{0,n+1},$$ 
sending $[(x_1,\cdots,x_n)]$ to $[(\infty,x_1,\cdots,x_n)]$, is a bijection.
\end{lemma}
\begin{proof}
Suppose $\mathcal{I}([(x_1,\cdots,x_n)])=\mathcal{I}([(x'_1,\cdots,x'_n)])$. Then there is an $M\in\mathrm{PSL}_2(\mathbb{C})$ such that 
$$(M(\infty),M(x_1),\cdots, M(x_n))=(\infty,x'_1,\cdots,x'_n).$$
Thus $M$ is affine and $M(x_i)=x'_i$. Hence $[(x_1,\cdots,x_n)]=[(x'_1,\cdots,x'_n)]$. So $\mathcal{I}$ is injective.\par 
For any point $[(p_1,\cdots,p_{n+1})]\in\mathcal{M}_{0,n+1}$, there exists a  representative has form $(\infty,p_2',\cdots,p'_{d+1})$. So $\mathcal{I}$ maps $[(p_2',\cdots,p'_{d+1})]\in\mathcal{M}^\ast_{0,n}$ to $[(\infty,p_2',\cdots,p'_{d+1})]\in\mathcal{M}_{0,n+1}$. Thus, $\mathcal{I}([(p_2',\cdots,p'_{d+1})])=[(p_1,\cdots,p_{n+1})]$. So $\mathcal{I}$ is surjective.
\end{proof}
Note $\mathcal{M}_{0,n+1}$ has Deligne-Mumford compactification $\widehat{\mathcal{M}}_{0,n+1}$. Then the map $\mathcal{I}$ induces a compactification of the space $\mathcal{M}^\ast_{0,n}$. We denote this compactification by 
$$\widehat{\mathcal{M}}^\ast_{0,n}:=\widehat{\mathcal{M}}_{0,n+1}.$$

\subsection{Marked Newton Maps}
In this subsection, we consider spaces of the fixed-points marked Newton maps and the corresponding moduli spaces.\par
Recall that a fixed-point marked rational map 
$$(f, p_1,\cdots,p_{d+1})\in\mathrm{Rat}_d\times(\hat{\mathbb{C}})^{d+1}$$ is a map $f\in\mathrm{Rat}_d$ together with an ordered list of its fixed points. Denote by $\mathrm{Rat}_d^{\mathrm{fm}}$ the space of all fixed-points marked degree $d$ rational maps. We can define the topology on $\mathrm{Rat}_d^{\mathrm{fm}}$ by convergence. We say $(f_n,p_1^{(n)},\cdots,p_{d+1}^{(n)})$ converges to $(f,d_1,\cdots,p_{d+1})\in\mathrm{Rat}_d^{\mathrm{fm}}$, as $n\to\infty$, if $f_n$ converges to $f$ uniformly and $p_i^{(n)}\to p_i$ for $i=1,\cdots,d+1$.\par 
Let $\mathrm{NM}_d^{\mathrm{fm}}\subset\mathrm{Rat}_d^{\mathrm{fm}}$ be the space of fixed-points marked degree $d$ Newton maps, and associate the space $\mathrm{NM}_d^{\mathrm{fm}}$ with the subspace topology. For $N\in\overline{\mathrm{NM}}_d\setminus\mathrm{NM}_d$, there exists a unique set $r=\{r_1,\cdots,r_d\}$ of $d$ points (not necessary distinct) in $\mathbb{P}^1$ such that $N=N_r$. We define $r_1,\cdots,r_d,\infty$ to be the fixed points of the degenerate Newton maps $N_r$. Let $\overline{\mathrm{NM}}_d^{\mathrm{fm}}$ be the space consisting of $(N,p_1\cdots,p_{d+1})$ where $N\in\overline{\mathrm{NM}}_d$ and $(p_1\cdots,p_{d+1})$ is an ordered list of the fixed points of $N$. In fact, for $(N=H_N\widehat N,p_1\cdots,p_{d+1})\in\overline{\mathrm{NM}}_d^{\mathrm{fm}}$, if $r\in\mathbb{P}^1$ is a fixed point of $\widehat N$, there are $d_{r}(N)+1$ many points, counted multiplicity, in $\{p_1,\cdots,p_{d+1}\}$ are $r$, where $d_r(N)$ is the depth of $r\in\mathbb{P}^1$. On the space $\overline{\mathrm{NM}}_d^{\mathrm{fm}}$, we can define the topology by convergence. We say a sequence $\{(N_n,p_1^{(n)},\cdots,p_{d+1}^{(n)})\}$ in $\overline{\mathrm{NM}}_d^{\mathrm{fm}}$ converges to $(N,p_1,\cdots,p_{d+1})\in\overline{\mathrm{NM}}_d^{\mathrm{fm}}$, as $n\to\infty$, if $N_n$ converges to $N$ in $\mathbb{P}^{2d+1}$ and $p_i^{(n)}\to p_i$ for $i=1,\cdots,d+1$. Note $\mathrm{NM}_d^{\mathrm{fm}}\subset\overline{\mathrm{NM}}_d^{\mathrm{fm}}$. As a remark, we notice that this topology coincides with the subspace topology on the space $\mathrm{NM}_d^{\mathrm{fm}}$. \par 
\begin{lemma}\label{fixed-points-Newton-method}
Under the above topology, $\overline{\mathrm{NM}}_d^{\mathrm{fm}}$ is compact and contains $\mathrm{NM}_d^{fm}$ as an open dense subset.
\end{lemma}
\begin{proof}
Suppose $\{(N_n,p_1^{(n)},\cdots,p_{d+1}^{(n)})\}\subset\overline{\mathrm{NM}}_d^{\mathrm{fm}}$ is a sequence such that $(N_n,p_1^{(n)},\cdots,p_{d+1}^{(n)})\to(N,p_1,\cdots,p_{d+1})$, as $n\to\infty$. Then $N\in\overline{\mathrm{NM}}_d$ and each $p_i$ is a  fixed points of $N$. Thus, $(N,p_1,\cdots,p_{d+1})\in\overline{\mathrm{NM}}_d^{\mathrm{fm}}$. Note $\mathrm{NM}_d$ is a dense open subset of $\overline{\mathrm{NM}}_d$. Therefore, $\mathrm{NM}_d^{\mathrm{fm}}$ is an open dense subset of $\overline{\mathrm{NM}}_d^{\mathrm{fm}}$.
\end{proof}
Note $z=\infty$ is the unique repelling fixed point for each Newton maps. Thus it is reasonable to consider the the space $\mathrm{NM}_d^\ast\subset\mathrm{NM}_d^{\mathrm{fm}}$ consisting of $(N,\infty,p_1,\cdots,p_d)\in\mathrm{NM}_d^{\mathrm{fm}}$, where $p_i\in\mathbb{C}$. Similarly, define $\overline{\mathrm{NM}}_d^\ast\subset\overline{\mathrm{NM}}_d^{\mathrm{fm}}$ to be the closure of $\mathrm{NM}_d^\ast$.\par
The group $\mathrm{PSL}_2(\mathbb{C})$ can act on $\mathrm{Rat}_d^{\mathrm{fm}}$ by right action:
$$g\cdot(f, p_1,\cdots,p_{d+1})=(g^{-1}\circ f\circ g, g^{-1}(p_1),\cdots,g^{-1}(p_{d+1})).$$
Define the fixed-points marked moduli space $\mathrm{rat}_d^{\mathrm{fm}}$ to be the quotient space 
$$\mathrm{rat}_d^{\mathrm{fm}}:=\mathrm{Rat}_d^{\mathrm{fm}}/\mathrm{PSL}_2(\mathbb{C}).$$\par 
By the specialty of $z=\infty$, let 
$$\mathrm{nm}_d^\ast:=\mathrm{NM}_d^\ast/\mathrm{Aut}(\mathbb{C})\subset \mathrm{rat}_d^{\mathrm{fm}}$$
be the fixed-points marked moduli space of degree $d$ Newton maps, associated with the quotient topology. \par
Let $r=\{r_1,\cdots,r_d\}$ and $\tilde{r}=\{0,1,\tilde{r}_3,\cdots,\tilde{r}_d\}$ be two subsets of $d$ distinct points in $\mathbb{C}$. If there exists $M\in\mathrm{Aut}(\mathbb{C})$ such that $M(0)=r_1$, $M(1)=r_2$ and $M(\tilde{r}_i)=r_i$ for $3\le i\le d$, then 
$$M^{-1}\circ N_r\circ M=N_{\tilde{r}}.$$ 
Thus
$$[(N_r,\infty,r_1,\cdots, r_d)]=[(N_{\tilde{r}},\infty,0,1, \tilde{r}_3,\cdots, \tilde{r}_d)]\in\mathrm{nm}_d^\ast.$$
Then we have 
\begin{lemma}
For $d\ge 2$, let $r=\{0,1,r_3,\cdots,r_d\}\subset\mathbb{C}$ be a set of $d$ distinct points. Then, as set
$$\mathrm{nm}_d^\ast=\{[(N_r,\infty,0,1,r_3,\cdots,r_d)]:(r_3,\cdots,r_d)\in(\mathbb{C}\setminus\{0,1\})^{d-2}\setminus\Delta_{d-2}\}.$$
\end{lemma}
Furthermore, 
\begin{proposition}\label{Newton-method-affine-moduli}
The moduli space $\mathrm{nm}_d^\ast$ is homeomorphic to the moduli space $\mathcal{M}^\ast_{0,d}$.
\end{proposition}
\begin{proof}Consider the following diagram
\[
\begin{CD}
\mathrm{NM}_d^\ast@>\widetilde\phi>>\mathbb{C}^{d}\setminus\Delta_{d}\\
@V/\mathrm{Aut}(\mathbb{C})VV@V/\mathrm{Aut}(\mathbb{C})VV\\
\mathrm{nm}^\ast@>\phi>>\mathcal{M}^\ast_{0,d},
\end{CD}
\]
where 
$$\widetilde\phi((N,\infty,r_1,\cdots,r_{d}))=(r_1,\cdots,r_{d})$$ and 
$$\phi([(N,\infty,r_1,\cdots,r_{d})])=[(r_1,\cdots,r_{d})].$$ 
Then the diagram is commutes.\par
Note 
$$(N,\infty,r_1,\cdots,r_d)\in\mathrm{NM}_d^\ast\subset\mathrm{Rat}_d\times\{\infty\}\times\mathbb{C}^d\setminus\Delta_d$$ 
and $N$ is uniquely determined by $(r_1,\cdots,r_{d})$.
Thus $\widetilde\phi$ is bijective, open and continuous. Hence, $\widetilde\phi$ is a homeomorphism. Therefore, $\phi$ is a homeomorphism.
\end{proof}
By Proposition \ref{Newton-method-affine-moduli}, we know the compactification $\widehat{\mathcal{M}}^\ast_{0,d}$ of $\mathcal{M}^\ast_{0,d}$ induces a compactification of $\mathrm{nm}^\ast_d$. We say this compactification is the Deligne-Mumford compactification of $\mathrm{nm}_d$ and denoted by $\widehat{\mathrm{nm}}^\ast_d$.\par
\begin{example}
Cubic Newton maps.\par
Note $\mathrm{nm}_3$ is homeomorphic to $\mathcal{M}^\ast_{0,3}\cong\mathbb{C}\setminus\{0,1\}$. Thus $\widehat{\mathcal{M}}^\ast_{0,3}\cong\mathbb{P}^1$. So $\widehat{\mathrm{nm}}^\ast_3$ is homeomorphic to $\mathbb{P}^1$.
\end{example}
\begin{example}\label{quartic-DM-dim}
Quartic Newton maps.\par
Note $\mathrm{nm}_4$ is homeomorphic to $\mathcal{M}^\ast_{0,4}$. By definition, if $[x_1,x_2,x_3,x_4]\in\widehat{\mathcal{M}}^\ast_{0,4}\setminus\mathcal{M}^\ast_{0,4}$, then $[\infty,x_1,x_2,x_3,x_4]\in\widehat{\mathcal{M}}_{0,5}\setminus\mathcal{M}_{0,5}$. Thus $\widehat{\mathrm{nm}}^\ast_4\setminus\mathrm{nm}_4$ has complex dimension $1$.
\end{example}

\section{From Deligne-Mumford to GIT}\label{DM to GIT}
In this section, we study the relationships between the Deligne-Mumford compactification $\widehat{\mathrm{nm}}^\ast_d$ and the GIT compactification $\overline{\mathrm{nm}}_d$, and prove Theorem \ref{theorem-Deligne-Mumford-GIT}. Then main ingredient is to construct the marked Berkovich trees of spheres.\par
Let $\mathbb{L}$ be the completion of the field of formal Puiseux series and let $\mathbb{P}^1_{\mathrm{Ber}}$ be the Berkovich space over $\mathbb{L}$. For more background of Berkovich space, we refer \cite{Baker10}. For $t\in\mathbb{D}$, we say $\{f_t\}\subset\mathrm{Rat}_d$ be a holomorphic family if $f_t\in\mathrm{Rat}_d$ for $t\not=0$ and its coefficients are holomorphic functions of parameter $t$. The holomorphic family $\{f_t\}$ induces a rational map $\mathbf{f}\in\mathbb{L}(z)$ acting on $\mathbb{P}^1_{\mathrm{Ber}}$, see \cite{Kiwi14,Kiwi15} for details. Let $\{N_t\}$ be a holomorphic family of degree $d\ge 2$ Newton maps with superattracting fixed points $r(t)=\{r_1(t),\cdots,r_d(t)\}$. Regard each $r_i(t)$ as an element in $\mathbb{L}$ and denote by $\mathbf{r}_i$. Let $\mathbf{N}\in\mathbb{L}(z)$ be the induced map of $\{N_t\}$. Let $V$ be the set of type II repelling fixed points of $\mathbf{N}$, and define the following convex hulls in the Berkovich space $\mathbb{P}^1_{\mathrm{Ber}}$
$$H_{\mathrm{fix}}=\mathrm{Hull}(\{\mathbf{r}_1,\cdots,\mathbf{r}_d,\infty\})$$
and 
$$H_V=\mathrm{Hull}(V).$$
We refer \cite{Nie-1} for more details about these convex hulls.\par
If $\xi\in\mathbb{P}_{\mathrm{Ber}}^1$ is a type II point, then the tangent space $T_\xi\mathbb{P}_{\mathrm{Ber}}^1$ can be identified with $\mathbb{P}^1$. For $\vec{v}\in T_\xi\mathbb{P}_{\mathrm{Ber}}^1$, let $\mathbf{B}_\xi(\vec{v})^-$ be the Berkovich disk corresponding to $\vec{v}$. We say a Riemann sphere $\mathbb{P}^1$ with marked points is associated to $\xi$ if the marked points on $\mathbb{P}^1$ are induced by the directions $\vec{v}\in T_\xi\mathbb{P}_{\mathrm{Ber}}^1$ with $\mathbf{B}_\xi(\vec{v})^-\cap\{\mathbf{r}_1\cdots,\mathbf{r}_d,\infty\}\not=\emptyset$ via identifying $T_\xi\mathbb{P}_{\mathrm{Ber}}^1$ with $\mathbb{P}^1$. Note $H_V$ is a subtree in $\mathbb{P}_{\mathrm{Ber}}^1$. Now we associated $(\mathbf{N},\infty,\mathbf{r}_1,\dots,\mathbf{r}_d)$ a marked Berkovich tree $\mathbf{T}$ of spheres .  
\begin{definition}
The \textit{marked Berkovich tree of spheres} $\mathbf{T}$ for $(\mathbf{N},\infty,\mathbf{r}_1,\dots,\mathbf{r}_d)$ is the data of:
\begin{enumerate}
\item the subtree $H_V\subset\mathbb{P}_{\mathrm{Ber}}^1$,
\item for each point $v\in V$, take the associated marked Riemann sphere $\mathbb{P}^1$ with $v$. 
\end{enumerate}
\end{definition}
\begin{remark}
In general, the marked Berkovich trees of spheres can be defined for any $n\ge 3$ distinct marked points in $\mathbb{L}$ since we can analog $H_V$ to define the corresponding convex hull for these $n$ points. We refer \cite{Arfeux15} for such trees.
\end{remark}
If $H_V$ is not trivial, we can forget the information from the edges of $H_V$ and attach the Riemann spheres in $\mathbf{T}$ by identifying marked points in spheres associated to vertices incident to a common edge in $H_V$. Then we get a Riemann surface $\mathbf{\Sigma}$ with nodes. We say $\mathbf{\Sigma}$ is the Riemann surface with nodes induced by $\mathbf{T}$. For convenience, if $H_V$ is trivial, denote $\mathbf{\Sigma}=\mathbf{T}$, which is a marked Riemann sphere. 
\begin{example}
Cubic Newton maps.\par
Consider $r(t)=\{0,1,t\}$. Then $V=\{\xi_{0,|t|},\xi_g\}$, hence $H_V=[\xi_{0,|t|},\xi_g]$. So $\mathbf{T}$ is a tree of two Riemann spheres. Note the $\mathbb{P}^1$ at $\xi_{0,|t|}$ has two marked points $0$ and $1$, and the $\mathbb{P}^1$ at $\xi_g$ has two marked points $1$ and $\infty$. Hence $\mathbf{\Sigma}$ is a Riemann surface with $4$ marked points and $1$ node.
\end{example}
We say two marked Berkovich trees $\mathbf{T}$ and $\mathbf{T}'$ of spheres  are equivalent if they induce equivalent Riemann surfaces $\mathbf{\Sigma}$ and $\mathbf{\Sigma}'$ with nodes, that is $[\mathbf{\Sigma}]=[\mathbf{\Sigma}']$ in $\widehat{\mathcal{M}}^\ast_{0,d}$. Denote by $[\mathbf{T}]$ for the equivalent class. Let $\mathcal{T}_d$ be the set consisting of the equivalent classes $[\mathbf{T}]$.
\begin{proposition}\label{Berkovich-tree-moduli-space}
For $d\ge 2$, the map 
$$\phi:\mathcal{T}_d\to\widehat{\mathcal{M}}^\ast_{0,d},$$ 
sending $[\mathbf{T}]$ to $[\mathbf{\Sigma}]$, where $\mathbf{\Sigma}$ is the Riemann surface with nodes induced by $\mathbf{T}$, is a bijection.
\end{proposition}
\begin{proof}
First by the definition of equivalence in Berkovich trees of spheres, we know the map $\phi$ is well defined. For $[\Sigma]\in\widehat{\mathcal{M}}^\ast_{0,d}$, there exits $\Sigma_t=(\mathbb{P}^1,\infty, r_1(t),\cdots,r_d(t))$ such that $\Sigma_t\to[\Sigma]$ as $t\to 0$, where $r_i(t)$s are holomorprhic functions of $t$. Let $\mathbf{r}_i:=r_i(t)\in\mathbb{L}$ and $\mathbf{r}=\{\mathbf{r}_1,\cdots,\mathbf{r}_d\}$. Then we can obtain a marked Berkovich tree $\mathbf{T}$ of spheres. With a suitable order of marked points we know the Riemann surface with nodes induced by $\mathbf{T}$ is $\Sigma$. Thus $\phi$ is surjective.\par
Now we show $\phi$ is injective. If $[\mathbf{\Sigma}]=[\mathbf{\Sigma}']$, let $\mathbf{T}$ and $\mathbf{T}'$ be the Berkovich trees of spheres who induce $\mathbf{\Sigma}$ and $\mathbf{\Sigma}'$, respectively. Then $[\mathbf{T}]=[\mathbf{T}']$. So $\phi$ is injective. 
\end{proof}
We say a sequence $\{[\mathbf{T}_n]\}$ in $\mathcal{T}_d$ converges to $[\mathbf{T}]\in\mathcal{T}_d$ if $\phi([\mathbf{T}_n])$ converges to $\phi([\mathbf{T}])$ in $\widehat{\mathcal{M}}^\ast_{0,d}$. Then we associate  $\mathcal{T}_d$ with the topology induced by this convergence. Under this topology on $\mathcal{T}_d$, the map $\phi$ is a homeomorphism. Let $\mathcal{S}_d\subset\mathcal{T}_d$ be the subset consisting of the equivalence classes $[\mathbf{T}]$, where $\mathbf{T}$ is a marked Riemann surface, i.e. the tree is trivial. Then the map $\phi$ gives a map 
$$\phi|_{\mathcal{S}_d}:\mathcal{S}_d\to\mathcal{M}_{0,d}^\ast,$$
which is a homeomorphism. \par
\begin{example}
Let $r_n(t)=\{0,3t^2,(3+1/n)t^2)\}$ and $r(t)=\{0,3,3+t\}$. Let $\mathbf{N}_n$ and $\mathbf{N}$ be the associated maps for $\{N_{r_n(t)}\}$ and $\{N_{r(t)}\}$, respectively. Consider the corresponding Berkovich trees $\mathbf{T}_n$ of spheres for $(\mathbf{N}_n,\infty,0,3t^2,(3+1/n)t^2)$ and the tree $\mathbf{T}$ for $(\mathbf{N}_n,\infty,0,3,3+t)$. Then $[\mathbf{T}_n]$ converges to $[\mathbf{T}]$ in $\mathcal{T}_d$.
\end{example}
Now we relate the space $\mathcal{T}_d$ to the GIT compactification $\overline{\mathrm{nm}}_d$ of the moduli space $\mathrm{nm}_d$ of degree $d\ge 2$ Newton maps.\par
Let $\{f_t\}$ be a holomorphic family of degree $d\ge 2$ rational maps and let $\mathbf{f}$ be the associated map for $\{f_t\}$. For a type II point $\xi\in\mathbb{P}_{\mathrm{Ber}}^1$, we say $f\in\mathbb{P}^{2d+1}$ is the subalgebraic reduction of $\mathbf{f}$ at $\xi$ if $M_t^{-1}\circ f_t\circ M_t$ converges to $f$ in $\mathbb{P}^{2d+1}$,
where $\{M_t\}\subset\mathrm{PSL}_2(\mathbb{C})$ with the induced map $\mathbf{M}$ mapping $\xi_g$ to $\xi$. In fact, we can assume $\{M_t\}\subset\mathrm{Aut}(\mathbb{C})$. Moreover, we can define the subalgebraic reduction of $\mathbf{f}$ at any point in $\mathbb{H}_{\mathrm{Ber}}$ by considering the field extensions, see \cite[Section 4]{Faber13I}. For $\xi\in\mathbb{P}^1_{\mathrm{Ber}}$, denote by $\rho_{\xi}(\mathbf{f})$ the subalgebraic reduction of $\mathbf{f}$ at $\xi$. For a holomorphic family $\{N_t\}$ of degree $d\ge 2$ Newton maps, the subalgebraic reduction of the induced map $\mathbf{N}$ at a point $\xi\not\in H_V$ is not semistable. Indeed, if $\xi\not\in H_V^\infty$, where $H_V^\infty\subset\mathbb{P}^1_{\mathrm{Ber}}$ is the convex hull of $V\cup\{\infty\}$, then $\widehat{\rho_{\xi}(\mathbf{N})}$ is a constant \cite[Theorem 3.10]{Nie-1}. Note $\rho_{\xi}(\mathbf{N})$ is a degenerate Newton map. Then we know $\rho_{\xi}(\mathbf{N})$ has a unique hole at $\infty$. Hence $d_{\infty}\rho_{\xi}(\mathbf{N})=d$. If $\xi\in(\pi_{H_V}(\infty), \infty)$, where $\pi_{H_V}(\infty)$ is the projection from $\infty$ to $H_V$, then $\widehat{\rho_{\xi}(\mathbf{N})}$ is a linear map. In fact, in this case $\rho_{\xi}(\mathbf{N})$ has a unique hole at $a\in\mathbb{C}$, and hence $d_{a}\rho_{\xi}(\mathbf{N})=d-1$. So $\rho_{\xi}(\mathbf{N})$ is not semistable.   At the points in $H_V$, we have
\begin{lemma}\label{Berkovich-tree-semistable-reduction}
Let $\{N_t\}$ be a holomorphic family of degree $d\ge 2$ Newton maps and let $\mathbf{N}$ be the associated map for $\{N_t\}$. Then there exists a point $\xi\in V$ such that $\rho_\xi(\mathbf{N})$ is semistable. In particular, let 
$$H^{ss}_V=\{\xi\in H_V:\rho_\xi(\mathbf{N})\in\mathrm{Rat}_d^{ss}\}.$$
Then one of the following holds:
\begin{enumerate}
\item there is a unique point $\xi\in H_V$ such that $\rho_\xi(\mathbf{N})$ is stable. In this case, $H_V^{ss}=\{\xi\}$.
\item there is a point $\xi\in H_V$ such that $\rho_\xi(\mathbf{N})$ is semistable but not stable. In this case, $\rho_{\xi'}(\mathbf{N})$ is not stable for any $\xi'\in H_V$. Moreover, $H_V^{ss}$ is connected.
\end{enumerate}
\end{lemma}
\begin{proof}
We prove the case when $d$ is odd. Similar argument works for even $d$.\par
Suppose $\rho_\xi(\mathbf{N})$ is not semistable for all $\xi\in V$. Then By Lemma \ref{Newton-hole}, at any $\xi\in V$, $\rho_\xi(\mathbf{N})$ has a hole of depth at least $(d+1)/2$. We claim there exists $\xi\in V$ such that $\infty$ is a hole of $\rho_\xi(\mathbf{N})$ with depth at least $(d+1)/2$. Indeed, if the claim does not hold, then there exist an endpoint $\xi'$ of $H_V$ such that all the holes of $\rho_{\xi'}(\mathbf{N})$ have depth at most $(d-1)/2$, which means $\rho_{\xi'}(\mathbf{N})$ is semistable. Define
$$E_\infty=\{\xi\in V: d_\infty(\rho_{\xi}(\mathbf{N}))\ge\frac{d+1}{2}\}.$$
Then $V\setminus E_\infty\not=\emptyset$ since the visible point $\pi_{H_V}(\infty)\in V$ and the $\infty$ is not a hole of $\rho_{\pi_{H_V}(\infty)}(\mathbf{N})$. In fact, all endpoints of $H_V$, except $\pi_{H_V}(\infty)\in V$, are contained in $E_\infty$. For $\xi\in V\setminus E_\infty$, let $h_\xi\in\mathbb{C}$ be the unique hole of $\rho_\xi(\mathbf{N})$ with depth at least $(d+1)/2$ and let $\vec{v}_{\xi}\in T_\xi\mathbb{P}^1_{\mathrm{Ber}}$ be the direction corresponding to $h_\xi$. Now we pick $\xi\in V\setminus E_\infty$ such that $\mathbf{B}(\vec{v}_\xi)^-\cap V\subset E_\infty$ and let $\xi'\in E_\infty$ such that $(\xi',\xi)\cap V=\emptyset$. Since $d_{h_{\xi}}\rho_{\xi}(\mathbf{N})\ge (d+1)/2$, we know there are at least $(d+3)/2$ many $\mathbf{r}_i:=r_i(t)$s in $\mathbf{B}(\vec{v}_\xi)^-$. However, there are at least $(d+1)/2$ many $\mathbf{r}_i$s in $\mathbb{P}^1_{\mathrm{Ber}}\setminus\mathbf{B}(\vec{v}_\xi)^-$ since $d_{\infty}\rho_{\xi'}(\mathbf{N})\ge(d+1)/2$. It is a contradiction. Thus there exists $\xi\in V$ such that $\rho_\xi(\mathbf{N})$ is semistable.\par 
Now we suppose $\xi\in V$ such that $\rho_{\xi}(\mathbf{N})$ is stable. We show for any other points $\xi'\not=\xi$ in $H_V$, the subalgebraic reductions $\rho_{\xi'}(\mathbf{N})$ are not semistable. It is sufficient to show $\rho_{\xi'}(\mathbf{N})$ is not semistable for $\xi'\in V\setminus\{\xi\}$. Indeed, at any point $\xi''\in H_V\setminus V$, the subalgebraic reduction $\rho_{\xi''}(\mathbf{N})$ has exact two holes and one hole has depth at most $(d-3)/2$. Otherwise, let $h\in\mathbb{C}$ respond to the direction $\vec{v}\in T_\xi\mathbb{P}^1_{\mathrm{Ber}}$ with $\xi''\in\mathbf{B}_{\xi}(\vec{v})^-$. Then 
$$d_{h}(\rho_{\xi}(\mathbf{N}))\ge\frac{d-1}{2},$$ 
which contradicts with the stability of $\rho_{\xi}(\mathbf{N})$. Thus the other hole of $\rho_{\xi''}(\mathbf{N})$ has depth at least $(d-1)/2$, which means $\rho_{\xi''}(\mathbf{N})$ is not semistable. Now suppose there exists $\xi'\in V$ such that $\rho_{\xi'}(\mathbf{N})$ is semistable. Let $\vec{v}_1\in T_\xi\mathbb{P}^1_{\mathrm{Ber}}$ be the direction such that $\xi''\in\mathbf{B}_{\xi}(\vec{v}_1)^-$, and let $\vec{v}_2\in T_{\xi'}\mathbb{P}^1_{\mathrm{Ber}}$ be the direction such that $\xi\in\mathbf{B}_{\xi'}(\vec{v}_2)^-$. Note 
$$\mathbf{B}_{\xi}(\vec{v}_1)^-\cup\mathbf{B}_{\xi'}(\vec{v}_2)^-=\mathbb{P}^1_{\mathrm{Ber}}.$$
We may assume $\infty\in\mathbf{B}_{\xi}(\vec{v}_1)^-$. Since $\rho_{\xi}(\mathbf{N})$ is stable, there are at most $(d-3)/2$ many $\mathbf{r}_i$s in $\mathbf{B}_{\xi}(\vec{v}_1)^-$. Since $\rho_{\xi'}(\mathbf{N})$ is semistable, there are at most $(d+1)/2$ many $\mathbf{r}_i$s in $\mathbf{B}_{\xi'}(\vec{v}_2)^-$. Thus there are $d-1$ many $\mathbf{r}_i$s in $\mathbb{P}^1_{\mathrm{Ber}}$. It is a contradiction. Thus $\rho_{\xi'}(\mathbf{N})$ is not semistable.\par 
Suppose $\xi\in V$ such that $\rho_\xi(\mathbf{N})$ is semistable but not stable. Then by $(1)$, we know $\rho_{\xi'}(\mathbf{N})$ is not stable for any $\xi'\in H_V$. The remaining is to show the set $H_V^{ss}$ is connected. If $H_V^{ss}=\{\xi\}$, it is done. Now we suppose $\xi'\not=\xi$ in $H_V^{ss}$. We show the segment $[\xi,\xi']\subset H_V^{ss}$. If $\xi'$ is on an edge of $H_V$ attached $\xi$, then $[\xi,\xi']\subset H_V^{ss}$ since any subalgebraic reduction of $\mathbf{N}$ at a point in $(\xi,\xi')$ is the same as $\rho_{\xi'}(\mathbf{N})$. Now we suppose $\xi'$ is on an edge of $H_V$ attached $\xi$. We only need to show the case when $\xi'\in V$ and $(\xi,\xi')$ is an edge of $H_V$. Let $\xi''\in(\xi,\xi')$. Then $\rho_{\xi''}(\mathbf{N})$ has exact $2$ holes. Let $\vec{v},\vec{v}'\in T_{\xi''}\mathbb{P}^1_{\mathrm{Ber}}$ be the directions such that $\xi\in\mathbf{B}_{\xi''}(\vec{v})^-$ and $\xi'\in\mathbf{B}_{\xi''}(\vec{v}')^-$, respectively. Let $\vec{v}_\xi\in T_{\xi}\mathbb{P}^1_{\mathrm{Ber}}$ be the direction such that $\xi'\in\mathbf{B}_{\xi}(\vec{v}_\xi)^-$ and let $\vec{v}_{\xi'}\in T_{\xi'}\mathbb{P}^1_{\mathrm{Ber}}$ be the direction such that $\xi\in\mathbf{B}_{\xi'}(\vec{v}_{\xi'})^-$. Then we have 
$$\mathbf{B}_{\xi''}(\vec{v})^-\cap\{\mathbf{r}_1,\cdots,\mathbf{r}_d,\infty\}=\mathbf{B}_{\xi'}(\vec{v}_{\xi'})^-\cap\{\mathbf{r}_1,\cdots,\mathbf{r}_d,\infty\}$$
and 
$$\mathbf{B}_{\xi''}(\vec{v}')^-\cap\{\mathbf{r}_1,\cdots,\mathbf{r}_d,\infty\}=\mathbf{B}_{\xi}(\vec{v}_{\xi})^-\cap\{\mathbf{r}_1,\cdots,\mathbf{r}_d,\infty\}.$$
Let $h,h',h_1$ and $h_2\in\mathbb{C}$ correspond to the tangent vectors $\vec{v},\vec{v}',\vec{v}_{\xi}$ and $\vec{v}_{\xi'}$, respectively. Then by Lemma \ref{Newton-hole}, we have 
$$d_h(\rho_{\xi''}(\mathbf{N})=d_{h_2}(\rho_{\xi'}(\mathbf{N}))\le\frac{d-1}{2}$$ and 
$$d_{h'}(\rho_{\xi''}(\mathbf{N})=d_{h_1}(\rho_{\xi}(\mathbf{N}))\le\frac{d-1}{2}$$
since $\rho_{\xi}(\mathbf{N})$ and $\rho_{\xi'}(\mathbf{N})$ are semistable. Thus, again by Lemma \ref{Newton-hole},  we know $\rho_{\xi''}(\mathbf{N})$ is semistable. Hence $\xi''\in H_V^{ss}$. So $H_V^{ss}$ is connected.
\end{proof}
\begin{remark}
In \cite{Rumely15}, Rumely studied the minimal resultant locus $\mathrm{MinResLoc}(\phi)$ for a rational map $\phi\in K(z)$ over any arbitrary non-Archimedean field $K$. He proved $\mathrm{MinResLoc}(\phi)$ is either a singleton or a segment. Later, he gave characterizations of the set $\mathrm{MinResLoc}(\phi)$ by claiming that $\mathrm{MinResLoc}(\phi)$ consists of points $\xi\in\mathbb{P}^1_{\mathrm{Ber}}$ where $\phi$ has semistable subalgebraic reductions \cite[Theorem B]{Rumely17}. It is still unclear how to find out $\mathrm{MinResLoc}(\phi)$ precisely. In our case, for Newton maps $\mathbf{N}\in\mathbb{L}(z)$, Lemma \ref{Berkovich-tree-semistable-reduction} allows us to find  $\mathrm{MinResLoc}(\mathbf{N})$ in finitely many steps. 
\end{remark}
Now we give some examples to illustrate Lemma \ref{Berkovich-tree-semistable-reduction}.
\begin{example}
Cubic Newton maps.\par
Consider $r(t)=\{0,1,t\}$. Then the set $V=\{\xi_g, \xi_{0,|t|}\}$. At the Gauss point $\xi_g$, the subalgebraic reduction $\rho_{\xi_g}(\mathbf{N})$ is semistable. Moreover, for any point $\xi\in[\xi_{0,|t|},\xi_g]$, $\rho_{\xi}(\mathbf{N})$ is semistable. And there is no point $\xi$ such that $\rho_{\xi}(\mathbf{N})$ is stable.
\end{example}
\begin{example}
Quartic Newton maps.\par
Consider $r(t)=\{0,1,t^{-1},2t^{-1}\}$. Then the set $V=\{\xi_g, \xi\}$, where $\xi=\xi_{0,|t^{-1}|}$. At the Gauss point $\xi_g$, the subalgebraic reduction $\rho_{\xi_g}(\mathbf{N})$ is not (semi)stable. But at the point $\xi$, the subalgebraic reduction $\rho_{\xi}(\mathbf{N})$ is stable. Furthermore, $\xi$ is the only point where the subalgebraic reduction of $\mathbf{N}$ is (semi)stable.
\end{example}
\begin{example}
Quintic Newton maps.\par
Consider $r(t)=\{t,2t,1+t,1+2t,t^{-1}\}$. Then 
$$V=\{\xi_g, \xi_{0,|t|},\xi_{1,|t|},\xi_{0,|t^{-1}|}\}.$$ 
At the points $\xi_{0,|t|},\xi_{1,|t|}$ and $\xi_{0,|t^{-1}|}$, the subalgebraic reductions of $\mathbf{N}$ are not semistable. But at the point $\xi_g$, the subalgebraic reduction $\rho_{\xi_g}(\mathbf{N})$ is semistable but not stable. In fact, $\xi_g$ is the only point where the subalgebraic reduction of $\mathbf{N}$ is semistable.
\end{example}
Note the information about the edges of the subtree $H_V$ does not affect the resulted(semi)stable maps. Thus, from Lemma \ref{Berkovich-tree-semistable-reduction}, we can associate each element in $\mathcal{T}_d$ with a GIT equivalence class of (semi)stable maps. Combining with the topology on the space $\mathcal{T}_d$ and on the GIT compactification $\overline{\mathrm{nm}}_d$. We have 
\begin{proposition}\label{Berkovich-tree-GIT}
For $d\ge 2$, the map 
$$\psi:\mathcal{T}_d\to\overline{\mathrm{nm}}_d,$$ 
sending $[\mathbf{T}]$ to $[\rho_{\xi}(\mathbf{N})]_{\mathrm{GIT}}$, where $\xi\in V$ with the subalgebraic reduction $\rho_{\xi}(\mathbf{N})\in\mathrm{Rat}_d^{ss}$, is a continuous surjection.
\end{proposition}
Note $\psi(\mathcal{S}_d)=\mathrm{nm}_d$ and $\psi|_{\mathcal{S}_d}$ has degree $(d-2)!$. But on $\mathcal{T}_d\setminus\mathcal{S}_d$, the map $\psi$ may be an infinity-to-one map.
\begin{example}
Sextic Newton maps.\par
For $c\in\mathbb{C}\setminus\{0,1\}$, let $r_c(t)=\{0,1,c,t^{-1},2t^{-1},3t^{-1}\}$. Then $V=\{\xi_g,\xi_{0,|t^{-1}|}\}$. Note the subalgebraic reduction 
$$\rho_{\xi_{0,|t^{-1}|}}(\mathbf{N}_c)=N_{\{0,0,0,1,2,3\}}\in\mathrm{Rat}_6^s$$ 
is independent on $c\in\mathbb{C}\setminus\{0,1\}$. Note $[\mathbf{T}_c]\not=[\mathbf{T}_{c'}]$ if $c\not=c'$. But 
$$\psi([\mathbf{T}_c])=\psi([\mathbf{T}_{c'}])=[N_{\{0,0,0,1,2,3\}}]_{\mathrm{GIT}}\in\overline{\mathrm{nm}}_d.$$
\end{example}
\begin{example}
Newton maps of odd degrees $d\ge 3$.\par
By Proposition \ref{semi-singleton}, for every element $N\in\overline{\mathrm{NM}}_d\cap(\mathrm{Rat}_d^{ss}\setminus\mathrm{Rat}_d^s)$, we have 
$$[N]_{\mathrm{GIT}}=\big[X^\frac{d-1}{2}Y^\frac{d-1}{2}[(d-1)X:(d+1)Y]\big]_{\mathrm{GIT}}.$$
Thus, for a holomorphic family $\{N_t\}$ of Newton's method, if there exists $\xi\in V$ such that $\rho_{\xi}(\mathbf{N})\in\mathrm{Rat}_d^{ss}\setminus\mathrm{Rat}_d^s$, then 
$$\psi([\mathbf{T}])=\big[X^\frac{d-1}{2}Y^\frac{d-1}{2}[(d-1)X:(d+1)Y]\big]_{\mathrm{GIT}}.$$
If $d=3$, then there are three equivalent classes in $\mathcal{T}_3\setminus\mathcal{S}_3$ and all of them are mapped to $[XY[X:2Y]]_{\mathrm{GIT}}$ under $\psi$. If $d\ge 5$, then there are infinitely many elements in $\mathcal{T}_d\setminus\mathcal{S}_d$ that are mapped to $[X^{(d-1)/2}Y^{(d-1)/2}[(d-1)X:(d+1)Y]]_{\mathrm{GIT}}$ under $\psi$.
\end{example}
Combining with Propositions \ref{Berkovich-tree-moduli-space} and \ref{Berkovich-tree-GIT}, we have 
\begin{theorem}
For $d\ge 2$, there is a natural continuous surjection 
$$\kappa:\widehat{\mathcal{M}}^\ast_{0,d}\to\overline{\mathrm{nm}}_d$$
such that $\kappa|_{\mathcal{M}^\ast_{0,d}}$ is $(d-2)!$-to-$1$. 
\end{theorem}
\begin{proof}
Set $\kappa=\psi\circ\phi^{-1}$, where $\phi$ and $\psi$ as in Propositions \ref{Berkovich-tree-moduli-space} and \ref{Berkovich-tree-GIT}, respectively.
\end{proof}

\section*{Acknowledgements}
This work is part of the author's PhD thesis. He would like to thank his advisor Kevin Pilgrim for fruitful discussions. The author also thanks Jan Kiwi, Laura DeMarco and Matthieu Arfeux for valuable comments.
\bibliographystyle{plain}
\bibliography{references}

\begin{thebibliography}{10}

\bibitem{Arfeux15}
M.~{Arfeux}.
\newblock {Berkovich spaces and Deligne-Mumford compactification}.
\newblock {\em ArXiv e-prints}, June 2015.

\bibitem{Baker10}
Matthew Baker and Robert Rumely.
\newblock {\em Potential theory and dynamics on the {B}erkovich projective
  line}, volume 159 of {\em Mathematical Surveys and Monographs}.
\newblock American Mathematical Society, Providence, RI, 2010.

\bibitem{Bers74}
Lipman Bers.
\newblock On spaces of {R}iemann surfaces with nodes.
\newblock {\em Bull. Amer. Math. Soc.}, 80:1219--1222, 1974.

\bibitem{Deligne69}
P.~Deligne and D.~Mumford.
\newblock The irreducibility of the space of curves of given genus.
\newblock {\em Inst. Hautes \'Etudes Sci. Publ. Math.}, (36):75--109, 1969.

\bibitem{DeMarco05}
Laura DeMarco.
\newblock Iteration at the boundary of the space of rational maps.
\newblock {\em Duke Math. J.}, 130(1):169--197, 2005.

\bibitem{DeMarco07}
Laura DeMarco.
\newblock The moduli space of quadratic rational maps.
\newblock {\em J. Amer. Math. Soc.}, 20(2):321--355, 2007.

\bibitem{Douady86}
Adrien Douady and Clifford~J. Earle.
\newblock Conformally natural extension of homeomorphisms of the circle.
\newblock {\em Acta Math.}, 157(1-2):23--48, 1986.

\bibitem{Faber13I}
Xander Faber.
\newblock Topology and geometry of the {B}erkovich ramification locus for
  rational functions, {I}.
\newblock {\em Manuscripta Math.}, 142(3-4):439--474, 2013.

\bibitem{Farkas92}
H.~M. Farkas and I.~Kra.
\newblock {\em Riemann surfaces}, volume~71 of {\em Graduate Texts in
  Mathematics}.
\newblock Springer-Verlag, New York, second edition, 1992.

\bibitem{Freire83}
Alexandre Freire, Artur Lopes, and Ricardo Ma\~n\'e.
\newblock An invariant measure for rational maps.
\newblock {\em Bol. Soc. Brasil. Mat.}, 14(1):45--62, 1983.

\bibitem{Funahashi12}
Risako Funahashi and Masahiko Taniguchi.
\newblock The cross-ratio compactification of the configuration space of
  ordered points on {$\widehat{\Bbb{C}}$}.
\newblock {\em Acta Math. Sin. (Engl. Ser.)}, 28(10):2129--2138, 2012.

\bibitem{Harris98}
Joe Harris and Ian Morrison.
\newblock {\em Moduli of curves}, volume 187 of {\em Graduate Texts in
  Mathematics}.
\newblock Springer-Verlag, New York, 1998.

\bibitem{Hubbard14}
John~H. Hubbard and Sarah Koch.
\newblock An analytic construction of the {D}eligne-{M}umford compactification
  of the moduli space of curves.
\newblock {\em J. Differential Geom.}, 98(2):261--313, 2014.

\bibitem{Kiwi14}
Jan Kiwi.
\newblock Puiseux series dynamics of quadratic rational maps.
\newblock {\em Israel J. Math.}, 201(2):631--700, 2014.

\bibitem{Kiwi15}
Jan Kiwi.
\newblock Rescaling limits of complex rational maps.
\newblock {\em Duke Math. J.}, 164(7):1437--1470, 2015.

\bibitem{Ljubich83}
M.~Ju. Ljubich.
\newblock Entropy properties of rational endomorphisms of the {R}iemann sphere.
\newblock {\em Ergodic Theory Dynam. Systems}, 3(3):351--385, 1983.

\bibitem{Mane83}
Ricardo Ma\~n\'e.
\newblock On the uniqueness of the maximizing measure for rational maps.
\newblock {\em Bol. Soc. Brasil. Mat.}, 14(1):27--43, 1983.

\bibitem{Milnor93}
John Milnor.
\newblock Geometry and dynamics of quadratic rational maps.
\newblock {\em Experiment. Math.}, 2(1):37--83, 1993.
\newblock With an appendix by the author and Lei Tan.

\bibitem{Nie-1}
H.~{Nie}.
\newblock {Iteration at the Boundary of Newton Maps}.
\newblock {\em ArXiv e-prints}, March 2018.

\bibitem{Rumely15}
Robert Rumely.
\newblock The minimal resultant locus.
\newblock {\em Acta Arith.}, 169(3):251--290, 2015.

\bibitem{Rumely17}
Robert Rumely.
\newblock A new equivariant in nonarchimedean dynamics.
\newblock {\em Algebra Number Theory}, 11(4):841--884, 2017.

\bibitem{Salamon99}
Dietmar Salamon.
\newblock Lectures on {F}loer homology.
\newblock In {\em Symplectic geometry and topology ({P}ark {C}ity, {UT},
  1997)}, volume~7 of {\em IAS/Park City Math. Ser.}, pages 143--229. Amer.
  Math. Soc., Providence, RI, 1999.

\bibitem{Silverman98}
Joseph~H. Silverman.
\newblock The space of rational maps on {$\bold P^1$}.
\newblock {\em Duke Math. J.}, 94(1):41--77, 1998.

\bibitem{Stone79}
A.~H. Stone.
\newblock Inverse limits of compact spaces.
\newblock {\em General Topology Appl.}, 10(2):203--211, 1979.

\bibitem{Zvonkine12}
Dimitri Zvonkine.
\newblock An introduction to moduli spaces of curves and their intersection
  theory.
\newblock In {\em Handbook of {T}eichm\"uller theory. {V}olume {III}},
  volume~17 of {\em IRMA Lect. Math. Theor. Phys.}, pages 667--716. Eur. Math.
  Soc., Z\"urich, 2012.

\end{thebibliography}
\end{document}